\documentclass[12pt,a4paper,reqno]{article}
\usepackage{latexsym,amsfonts,amssymb,amsthm,upref,amscd}
\usepackage[T1]{fontenc}
\usepackage[reqno]{amsmath}

\usepackage[mathlines]{lineno}

\usepackage{changes}  
\definechangesauthor[name={Author1}, color=red]{A1} 
\definechangesauthor[name={Author2}, color=green]{A2} 

\usepackage{times}
\usepackage{mathrsfs}
\usepackage{constants}
\usepackage{booktabs}
\usepackage{verbatim}
\usepackage{pstricks,ifthen,calc}
\usepackage{array}
\usepackage[labelformat=simple,labelsep=colon]{caption}
\usepackage{xcolor}

\usepackage[breaklinks,colorlinks,linkcolor=red,
anchorcolor=blue,
bookmarks=true,
bookmarksopen=true,
citecolor=blue,
urlcolor=magenta]{hyperref}
\usepackage{xcolor}

\usepackage[misc]{ifsym}
\usepackage{fancyhdr}
\usepackage{geometry}
\usepackage{graphicx}
\usepackage{subfigure}

\usepackage[abbrev,lite,nobysame]{amsrefs}
\usepackage{authblk}

\newcommand{\xqedhere}[2]{%
	\rlap{\hbox to#1{\hfil\llap{\ensuremath{#2}}}}}

\renewcommand{\thefootnote}{}
\newtheorem{Theorem}{Theorem}[section]

\newtheorem{lemma}[Theorem]{Lemma}

\newtheorem{corollary}[Theorem]{Corollary}
\newtheorem{remark}[Theorem]{Remark}
\newtheorem{Open Problem}[Theorem]{Open Problem}

\makeatletter \@addtoreset{equation}{section} \makeatother
\newcommand{\abs}[1]{\lvert#1\rvert}

\newcommand{\st}{\,\vert\,}
\newcommand{\dif}{\,\mathrm{d}}
\newcommand{\R}{\mathbb{R}}

\newcommand{\cP}{\mathcal{P}}

\newcommand{\cB}{\mathcal{B}}

\newcommand{\cS}{\mathcal{S}}

\DeclareMathOperator{\dist}{dist}

\makeatletter
\def\thanks#1{\protected@xdef\@thanks{\@thanks
		\protect\footnotetext{#1}}}
\makeatother
\begin{document}
	\setcounter{footnote}{0}
	\renewcommand{\thefootnote}{\customsymbol{footnote}}
	
		\title{\bf  \Large Normalized groundstates for mixed $(p,2)$-Laplacian equations in $\mathbb R^2$ with exponential critical growth
	}
		\author[1]{\small Jiankang Xia$^\ast$\thanks{$^\ast$Corresponding author. E-mail address: jiankangxia@nwpu.edu.cn}}
	\author[2,1]{\small Chao Zhong$^{\dagger}$\footnote{$^{\dagger}$E-mail address: chaozhonghphz@mail.nwpu.edu.cn }}
	\affil[1]{{\footnotesize School of Mathematics and Statistics, Northwestern Polytechnical University, Xi'an 710129, China}}
	\affil[2]{{\footnotesize School of Mathematical Sciences, Beijing Normal University, Beijing, 100875, China}}
	\date{}
	\maketitle

	\vskip -.3 in
	\begin{minipage}{14cm}
			{\small {\bf Abstract:} We investigate normalized groundstates for  mixed $(p,2)$-Laplacian equations
				\begin{align*}
					\begin{cases}
						-\Delta_p u-\Delta u+\lambda u=f(u) & \text{in } \mathbb{R}^2,\\
						\displaystyle \int_{\mathbb{R}^2}|u|^2\,\mathrm{d}x=m,\\
						u\in H^1(\mathbb{R}^2)\cap D^{1,p}(\mathbb{R}^2),
					\end{cases}
				\end{align*}
				 where $\Delta_p$ denotes the $p$-Laplacian with $1<p<2$, $\lambda\in\mathbb{R}$ represents a Lagrange multiplier and the nonlinerity $f$ exhibits exponential critical growth. Compared to the single-Laplacian case, the lack of regularity here precludes the Pohozaev identity, and the exponential critical growth severely compromises the restoration of compactness. To address these issues, we introduce a refined Moser iteration technique adapted to exponential critical growth, which establishes the Pohozaev identity for weak solutions under the mere assumption of $C_{\mathrm{loc}}^{1,\alpha}$-regularity. By combining constrained minimization on the Pohozaev manifold within a closed $L^2$-ball with a minimax characterization, we prove the existence of normalized groundstates for any prescribed mass $m>0$. Notably, our approach works independently of the sign of the Lagrange multiplier $\lambda$, thereby surmounting the fundamental barrier in recovering compactness for mixed Laplacian problems.

			\medskip {\bf Key words:} Mixed $(p,2)$-Laplacian equations; Normalized groundstates; Exponential critical growth; Pohozaev identity
			
			\medskip {\bf 2020 Mathematics Subject Classification:} 35A15 $\cdot$ 35J92 $\cdot$  35B33  $\cdot$ 35M10 $\cdot$ 35B45
		}
	\end{minipage}
	
	\section{Introduction and main results}

	We are concerned with  the mixed $(p,2)$-Laplacian equation  
	\begin{equation}\label{2pnls}
		\begin{cases}
			-\Delta_p u	-\Delta u+\lambda u=f(u) \quad  \text{ in }\ \mathbb{R}^2 ,\\
			\displaystyle 	\int_{\mathbb{R}^2}|u|^2\dif x=m,
		\end{cases}
	\end{equation}
	where  $m>0$ is given in advance, $\Delta_pu:=\text{div}(|\nabla u|^{p-2}\nabla u)$ denotes the $p$-Laplacian with $1<p<2$, and $\lambda\in \mathbb R$ is part of the unknowns, while serving as a Lagrange multiplier.
	In recent years, the following quasilinear elliptic problem driven by the mixed $(p,q)$-Laplacian has attracted significant interest:
	\begin{equation}\label{eq1.2}
		-\Delta_p u-\Delta_q u + V(x)\abs{u}^{p-2}u + K(x)\abs{u}^{q-2} u=g(x, u) \quad \text{ in } \  \mathbb{R}^N,
	\end{equation}
	where $N\geq2$ and $1<p<q\leq N$. 
	Such stationary equations arise from the general reaction-diffusion system
	\begin{align*}
		\partial_t w=\text{div}(A(\nabla w)\nabla w) + c(x,w) \quad \text{with} \quad A(s)=|s|^{p-2}+|s|^{q-2},
	\end{align*}
	where $w(x,t)\in \mathbb R$ represents the density or concentration of multi-component substances, and $A(s)$ characterizes nonstandard diffusion processes. The inhomogeneous nonlinearity $c(x,w)$ denotes the reaction term incorporating sources and losses, which further accounts for spatially localized potentials or nonuniform media \cite{Fife79}. We also note that the $(p, q)$-Laplacian operator $-\Delta_{p} - \Delta_{q}$ constitutes a specific instance of the widely recognized double-phase operator
	\begin{equation*}
-{\rm div}\left(|\nabla u|^{p - 2}\nabla u + a(x)|\nabla u|^{q - 2}\nabla u\right),
	\end{equation*}
	where $a(x) \geq 0$ is bounded.  Such systems find extensive applications in various physical contexts, including biophysics, plasma physics, chemical reaction modeling; see e.g., \cite{CI05,Derrick60,Aris79}.
	The energy functional associated with this double-phase operator has been extensively studied both in homogenization and elasticity theory for modeling strongly anisotropic materials \cite{Zh86, Zh95}, as well as within the calculus of variations framework \cite{Mar93, MR21}. For further background, we
  refer to \cites{Mar90,Ma89} and the references therein.

	Regarding the unconstrained problem \eqref{eq1.2}, extensive studies have addressed cases both the potentials  $V$ and $K$  admit positive lower bounds; see for instance, \cites{AAM15,Fig11,PW2018, FP21, Amb23,Amb24} and the references therein. Specifically, Figueiredo \cite{Fig11} examined \eqref{eq1.2} with Sobolev critical growth, while Pomponio and Watanabe \cite{PW2018} established the existence of positive solutions for general nonlinearities. Furthermore, Fiscella and Pucci \cite{FP21} tackled the $(p,N)$-Laplacian equation in the exponential critical setting. More recently, Ambrosio \cite{Amb24} investigated  the $(p,q)$-Laplacian equation \eqref{eq1.2} with Berestycki--Lions type nonlinearities, including the limiting case where $q=N$.  For concentration results, we refer to \cites{Amb23,AF11}, while earlier contributions can be found in \cites{CI05,HL08}. Consequently, the natural function space for solutions to \eqref{eq1.2} is $W^{1,p}(\mathbb{R}^N) \cap W^{1,q}(\mathbb{R}^N)$.

	It is worthwhile noting that the zero-mass case, characterized by $V = K \equiv 0$ and $1 <p < q < N$, has been investigated in various studies, including \cites{BBF21,CEM15,FP21}. In this setting, the appropriate function space is $D^{1,p}(\mathbb{R}^N) \cap D^{1,q}(\mathbb{R}^N)$. For the limiting case $q= N$, the functional framework $E^{N,p}$ (see \eqref{combspace} below for details) was introduced in \cite{CFFM2021}, where the existence of solutions was established for the $(p,N)$-Laplacian equation in $\mathbb{R}^N$ with exponential critical growth. Furthermore, in the critical planar case $q=N=2$, the space $E^{2,p}$ was successfully employed in \cite{dCS23} to obtain the minimal energy solution for a Schr\"odinger--Poisson system with a logarithmic integral kernel. We also highlight the work of Liu and Perera \cite{LP24}, who established the multiplicity of solutions for $(p,q)$-Laplacian equations in both Sobolev subcritical and critical cases.

	When $p=q=2$, we are led to the well-understood nonlinear scalar field equation:
	\begin{equation}\label{eqnls}
		\begin{cases}
			-\Delta u +\lambda u =f(u) \quad  \text{ in } \  \mathbb{R}^N,\\
			\displaystyle	\int_{\mathbb{R}^N}\abs{u}^2\dif x=m.
		\end{cases}
	\end{equation}
	If $f$ exhibits $L^2$-subcritical growth (i.e., $f(u)=|u|^{r-2}u$ with $r<2+{4}/{N}$), the existence of a global minimizer follows from minimizing the constrained functional over the prescribed $L^2$-sphere, relying on the coercivity of the functional; see, e.g., \cites{Lions84,CL82}. In contrast, for the $L^2$-supercritical case, the constrained functional fails to be bounded from below. The study of normalized solutions in this regime dates back to the seminal work of Jeanjean \cite{Jeanjean97}, who constructed the mountain pass geometry. A key challenge in this context is proving the boundedness of the Palais--Smale sequence to restore compactness. To circumvent this, Jeanjean developed a variational approach by introducing a stretched functional in an augmented space, thus obtaining a bounded Palais--Smale sequence for which the Pohozaev identity  asymptotically satisfied.
 Recently, Hirata and Tanaka \cite{HT19} further advanced the augmented space theory by establishing a new deformation argument based on a revised Palais--Smale--Pohozaev sequence. We refer to \cites{MS22, MS24, JeanjeanLu2020} for more recent results on the single Laplacian equation \eqref{eqnls}. Notably, the Pohozaev identity, satisfied by all weak solutions, plays a crucial role in the search for normalized groundstates.
Here and in the sequel, ``normalized groundstate'' is understood as a  solution that has least energy among all solutions satisfying the $L^2$-constraint.
	
	In the planar case $p=N=2$, problem \eqref{eqnls} involves a critical nonlinearity with exponential growth due to the Trudinger--Moser inequality, which was first established in \cites{Trudinger67, Moser71} and subsequently advanced  in \cites{Cao92, dFdR02, adachitanaka00,dddS14}.
	We say that a function $f$ meets an exponential critical growth if there exists $\alpha_0 >0 $ such that
	\begin{equation}\label{assexpcri}
		\displaystyle \lim_{t\to \infty} \frac{\abs{f(t)}}{e^{\alpha\abs{t}^2}}=
		0, \quad  \forall \alpha >\alpha_0, \quad \text{and} \quad
		\displaystyle \lim_{t \to \infty} \frac{\abs{f(t)}}{e^{\alpha\abs{t}^2}}=	+\infty, \quad \forall \alpha <\alpha_0.
	\end{equation}
	In this situation, Alves and Ji \cite{AJM22} established the existence and multiplicity of normalized solutions to \eqref{eqnls} in  $\mathbb R^2$ by employing a minimax approach combined with the Trudinger--Moser inequality. Specifically, they assumed $\alpha_0=4\pi$ in \eqref{assexpcri} and  required that $f$ satisfies the Ambrosetti--Rabinowitz growth condition, which is essential for establishing the boundedness of the Palais--Smale sequence. Subsequently, Chang, Liu, and Yan \cite{CLY23} obtained normalized groundstates for any $m>0$ and $\alpha_0>0$ without imposing the Ambrosetti--Rabinowitz condition. Their approach involves a minimization argument on the prescribed closed $L^2$-ball, with references to the earlier work \cites{BM21,MS22} and early version of \cite{DHZ26}. A refined version of problem \eqref{eqnls} is available in a recent work  \cite{MW24}.
	Recently, Dou, Huang and Zhong \cite{DHZ24} demonstrated the existence of normalized solutions for the general $N$-Laplacian equation in $\mathbb R^N$ (for all dimensions $N\geq 2$) under exponential growth conditions and an $L^N$-constraint. 
	 For further results on unconstrained nonlinear scalar field equations with exponential critical growth, we refer to \cites{CQT24,dFMR95, dFR95, doS01, AS07, ASM12, Cao92} and references therein.
	
The following mixed $(p,q)$-Laplacian equation \eqref{pqnls} with $L^p$-constraint has received comparatively little attention:
	\begin{equation}\label{pqnls}
	\begin{cases}
		-\Delta_p u-\Delta_q u+\lambda\abs{u}^{p-2}u=f(u) \quad \text{ in }\  \mathbb{R}^{N},\\
		\displaystyle	\int_{\mathbb{R}^N}\abs{u}^p\dif x=m,\\
		u\in W^{1,p}(\mathbb{R}^N)\cap D^{1,q}(\mathbb{R}^N).
	\end{cases}
\end{equation}
This stands in stark contrast to the extensively studied single Laplacian case.
Baldelli and Yang \cite{BY25} investigated the miexd $(2,q)$ case ($1<q\neq 2<N$) with a homogeneous nonlinearity. More precisely, they established the existence of normalized groundstates in both the $L^2$-subcritical and $L^2$-supercritical regimes, and found infinitely many radial solutions in the latter case. 
Cai and  R\u adulescu  \cite{CR24} subsequently studied the general mixed $(p,q)$-Laplacian equation with an $L^p$ constraint for $1<p<q<N$, where $f$ satisfies a mass-supercritical growth condition; they established normalized groundstates via direct minimization on the Pohozaev manifold within a closed $L^p$-ball; a refinement under weaker assumptions was later provided in \cite{GT26}. Very recently, Huang, Luo and Wang \cite{HLW26} extended the results of \cite{BY25} to the inhomogeneous case, while Ding, Ji and Pucci \cite{DJP25,DJP26} revealed existence and multiplicity results for $(2,q)$-Laplacian equation with $L^2$-constraints, allowing the nonlinearity $f$ to exhibit either a strongly sublinear regime near the origin or a combination of mass-subcritical growth at the origin and mass-supercritical growth at infinity.

	Inspired by the aforementioned literature, we focus on the existence and qualitative properties of normalized solutions to problem \eqref{2pnls}, where the nonlinearity exhibits exponential critical growth. In particular, we emphasize on the effects of the quasilinear operator $-\Delta_p$. As is well known, normalized solutions can be obtained by identifying critical points of the following functional:
	\begin{equation}\label{enerfunct}
		I(u)=\frac{1}{2}\int_{\mathbb{R}^2}\abs{\nabla u}^2 \dif x+\frac{1}{p}\int_{\mathbb{R}^2}|\nabla u|^p \dif x-\int_{\mathbb{R}^2}F(u)\dif x
	\end{equation}
constrained on the $L^2$-sphere
	\begin{align*}
		\mathcal{S}_m=\big\{u\in E : \|u\|_{L^2(\R^2)}^2=m\big\},
	\end{align*}
	where $F(u):=\int_{0}^{u}f(s)\dif s$ and $E:=H^1(\mathbb{R}^2)\cap D^{1,p}(\mathbb{R}^2)$ is endowed with the norm
	\begin{equation*}\|u\|_E :=\big(\|\nabla u\|_{L^2(\R^2)}^2+\|u\|_{L^2(\R^2)}^2+\|\nabla u\|_{L^p(\R^2)}^2\big)^{1/2}.
	\end{equation*}
	
	It is widely recognized that the Pohozaev manifold, denoted by $\mathcal{P}$ and defined precisely in \eqref{Pohozaevmf} below, plays a fundamental role in characterizing groundstates.
	A natural and effective strategy is to minimize the functional restricted to the Pohozaev manifold $\mathcal{P}$ within $\cS_m$, see e.g., \cites{JeanjeanLu2020,Jeanjean97,MS22}.
	The validity of this approach rests on the fact that all normalized solutions of \eqref{eqnls} necessarily lie in $\mathcal{P}_m=\cS_m\cap \cP$.
	 However, for quasilinear elliptic problems, it remains an open question whether the corresponding Pohozaev identity holds, primarily due to the limited regularity of weak solutions (see e.g., \cites{HeLi08,Ma89}). In fact, for problem \eqref{2pnls}, we can derive a $C_{\mathrm{loc}}^{1,\gamma}$ estimate for some $\gamma \in (0,1)$, rather than the higher $C^{2}$ or $C^{2,\gamma}$-regularity, see Theorem \ref{phothm}.
	Although Ambrosio \cite{Amb24} has established  the Pohozaev identity for the limiting case in the context of subcritical exponential nonlinearities, those findings are not applicable to the problem we are considering. This is attributed to the constrained nature of our problem, where the Lagrange multiplier $\lambda>0$ cannot be  predetermined---a factor that is crucial to Ambrosio's proof.

	Nevertheless, we are still able to establish the Pohozaev identity. Our strategy involves first deriving a local $L^\infty_{\text{loc}}(\mathbb{R}^2)$ estimate for the solutions rather than a global $L^\infty(\mathbb{R}^2)$ estimate. This step draws inspiration from a similar Moser iteration procedure described by de Albuquerque, Carvalho, and Silva in \cite{dCS23} (see also \cite{Amb24}), which enables us to obtain $C^{1,\gamma}_{\text{loc}}$-regularity of the solutions by adapting the proof from He and Li \cite{HeLi08}*{Theorem 1}. Finally, by employing techniques established by Degiovanni, Musesti, and Squassina in \cite{DMS03}, which allow us to avoid interchanging the second derivatives in the Hessian matrix of solutions, we successfully derive the corresponding Pohozaev identity, even in the absence of sufficient regularity. We emphasize that our approach is highly flexible and does not require any sign information regarding $\lambda$. This feature enables straightforward adaptation with only minor modifications, to 
	$N$-Laplacian or $(p,N)$-Laplacian equations with exponential critical growth nonlinearity. We believe this aspect holds significant independent  interest.

	\begin{Theorem}\label{phothm}
		Assume that $f$ satisfies \eqref{assexpcri} and $f(t)=o(t)$ as $t\to 0$. If $u\in E$ is a weak solution of  Problem \eqref{2pnls} with $\lambda\in \R$ being fixed, then $u\in C^{1,\gamma}_{{\rm loc}}(\mathbb R^2)$ for some $\gamma\in (0,1)$. Moreover, $u$ satisfies the Pohozaev identity:
		\begin{align*}
			\frac{2-p}{p}\int_{\mathbb R^2}|\nabla u|^p \dif x+\lambda\int_{\mathbb R^2}|u|^2  \dif x=2\int_{\R^2}F(u)\dif x.
		\end{align*}
	\end{Theorem}
	
	Let $H(s):=f(s)s-2F(s)$. We now introduce the Pohozaev functional $P: E \to \mathbb{R}$ as follows:
	\begin{align*} 
		P(u):=\int_{\mathbb R^2}|\nabla u |^2 \dif x+2\big(1-\frac{1}{p}\big)\int_{\mathbb R^2}|\nabla u|^p \dif x -\int_{\mathbb{R}^2}H(u)\dif x.
	\end{align*}
	We then define the corresponding manifolds $\mathcal{P}$ and $\mathcal{P}_m$ by,
	\begin{align}\label{Pohozaevmf}
		\mathcal{P}:=\{u\in E\setminus\{0\}: P(u)=0\}, \quad \mathcal{P}_m:=\{u\in \mathcal{S}_m: P(u)=0\},
	\end{align}
	
	Given our focus on the normalized groundstates, we  are naturally led to consider the minimization problem
	\begin{equation}\label{minipoho}
		\gamma_m:=\inf_{u\in\cP_m}I(u).
	\end{equation}
The monotonicity of $\gamma_m$ with respect to $m$ plays a crucial role in restoring $L^2$-compactness in the mass-supercritical regime; see, e.g., \cites{JeanjeanLu2020,Jeanjean97}. However, this property is difficult to establish under exponential critical growth, as it typically requires the inequality $f(t)t\leq p^*F(t)$ to hold, where $p^*:={Np}/(N-p)$ is the Sobolev critical exponent associated with $-\Delta_p$. Unfortunately, this equality fails in the exponential critical setting.

Adopting a strategy analogous to that employed in \cites{BM21,MS22,DHZ26}, we search for minimizers of  $I$ on $\cB_m \cap \mathcal{P}$  rather than $\cS_m\cap\mathcal{P}$, with $\cB_m$ being the closed $L^2$-ball 
	\begin{align*}
		\mathcal{B}_m:=\big\{u\in E\setminus\{0\}: \|u\|_{L^2(\R^2)}^2\leq m\big\}.
	\end{align*}
Accordingly, we  formulate the minimization problem
	\begin{equation}\label{minipro}
		\varUpsilon_m:=\inf\limits_{u\in\mathcal{B}_m\cap\mathcal{P}}I(u).
	\end{equation}
	
	To state our  result, we impose the following assumptions on the nonlinearity $f$:
	\begin{enumerate}
		\item [$(f_1)$] $f \in C^1(\mathbb{R},\mathbb{R})$ is odd and $f(t)=o(t^{3})$ as $t \to 0$;
		\item [$(f_2$)] there exists $\alpha_0 >0 $ such that $f$ satisfies \eqref{assexpcri};
		\item [($f_3$)] for all $t\in \mathbb R$, there holds $h(t)t \geq 4H(t)$ where $h(t):= H'(t)$;
		\item [($f_4$)] there exist $\xi > 0$ and $\nu > 4$ such that
		\begin{equation*}
			\operatorname{sgn}(t) f(t) \geq \xi \abs{t}^{\nu-1}, \quad \forall t \in \mathbb{R},
		\end{equation*}
		with the notation $\operatorname{sgn}(\cdot)$ being the sign function.
	\end{enumerate}
	\begin{Theorem}\label{thm1}
		Assume that $1<p<2$ and  $f$ satisfies $(f_1)$--$(f_4)$. Then
		for any $m>0$, there exists ${\xi}_0>0$ such that Problem \eqref{2pnls} admits a normalized groundstate $u\in E$ for some $\lambda >0$ whenever  $\xi\geq {\xi}_0$. In particular, $u$ is radially symmetric and satisfies 
		\begin{equation*}
			I(u)=\varUpsilon_m=\gamma_m.
		\end{equation*}
	\end{Theorem}

In contrast to the single Laplacian equation \eqref{eqnls} with exponential critical nonlinearity, the existence of normalized groundstates for Problem \eqref{2pnls} can still be established for arbitrary $\alpha_0 > 0$ and $m > 0$ under the same exponential critical growth conditions considered in \cite{CLY23}.
This strengthened result underscores the dominant role of the $2$-Laplacian operatorn and broadens the applicability of our findings. As evidenced by our analysis, the $\Delta_p$ operator has no effect on the existence of solutions in our $L^2$-supercritcal setting near origin and exponential critical growth at infinity, confirming the decisive role of the $2$-Laplacian operator. 
This is reminiscent of the recent work \cite{ZZL25}, where the $L^q$-constraint matches the larger exponent $q$ of the mixed $(2,q)$-Laplacian operator ($2<q<3=N$), leading to the dominant role of the $q$-Laplacian operator.

Unlike the single Laplacian setting, the main analytical challenge arises from simultaneously handling the quasilinear term and exponential critical nonlinearity.
To address these issues, we restrict our analysis to the function space 
 $H^1(\mathbb{R}^2) \cap D^{1,p}(\mathbb{R}^2)$,
	which guarantees the continuous embedding of $E^{2,p}$ into $L^r(\mathbb{R}^2)$ for $r>2p/(2-p)$ (see Lemma \ref{lem2.2} below). This embedding enables us to apply a new version of Trudinger--Moser inequality, see Lemma \ref{lem2.1}. Following \cites{BM21,MS22,DHZ26}, we solve the minimization problem \eqref{minipro} on $\mathcal{B}_m \cap \mathcal{P}$, and the solution minimizes on $\cP_m=\mathcal{S}_m \cap \mathcal{P}$ as well, see Lemma \ref{equalvalues}. This ensures the existence of $u \in \mathcal{S}_m$ and Lagrange multipliers $\lambda, \mu \in \mathbb{R}$ satisfying
\begin{equation*}
	I'(u)+ \lambda u + \mu P'(u)=0.
\end{equation*}
	However, it remains challenging to prove $\mu = 0$, where $\mu$ denotes the Lagrange multiplier corresponding to the constraint $P(u)=0$. This difficulty stems from the presence of the quasilinear operator $\Delta_p$, which in turn results in a lack of information regarding the sign of $\lambda$. To overcome this, we adopt a technique from \cites{SW1,SW2} extended to the constrained framework in \cite{JeanjeanLu2020}, which enables us to establish a minimax characterization of $\varUpsilon_m$ and construct a Palais--Smale sequence in $\cP_m$ via the dual variational principle. This specially designed minimizing sequence is in fact a Palais--Smale sequence for the energy functional constrained on the $L^2$-sphere. As we will show, this sequence exhibits strong convergence, from which we can deduce the existence of a groundstate without explicitly invoking the Lagrange multiplier $\mu$ or knowing the sign of $\lambda$. Here $\lambda$ denotes the associated Lagrange multiplier with respect to the $L^2$-constraint.
	
	The paper is organized as follows. 
	In Section~\ref{sect2},  we provide some preliminaries and derive the 
$C_{\mathrm{loc}}^{1,\alpha}$-regularity for weak solutions of \eqref{2pnls}, then prove Theorem \ref{phothm} by establishing the Pohozaev identity under the sole assumption of $ C_{\mathrm{loc}}^{1,\alpha} $ -regularity. Section~\ref{sect3} is devoted to analyzing the asymptotic behavior of the groundstate energy level $\varUpsilon_{m}$, establishing an upper bound for the minimizing sequences associated with  $\varUpsilon_{m}$, and demonstrating that $\varUpsilon_m=\gamma_m$ for largely $\xi$.
	Finally, in Section~\ref{sect4}, we establish the minimax characterization of $\varUpsilon_m$ for largely $\xi$ and complete the proof of Theorem~\ref{thm1}.

	\vskip 0.1in

	\textbf{Notations}:	We denote by $C_0$, $C_1$, $...$ the generic positive constants, which may vary from line to line.    $\|\cdot\|_r$ stands for the  standard norm in $L^r(\mathbb R^2)$ ($1\leq r<\infty$).
	Finally, $o_n(1)$ means that $o_n(1)\to 0$ as $n\to \infty$.

	\section{Regularity and Pohozaev identity}
	\label{sect2}
	
	For $r\in (1,N)$, we  denote by $D^{1,r}(\mathbb R^N)$ the closure of $C^\infty_c(\mathbb R^N)$, the space of smooth functions with compact support in $\mathbb{R}^N$, with respect to the norm $\|\nabla u\|_r$.
	We also recall the function space $E^{N,r}$, which is defined as the completion of $C_c^{\infty}(\mathbb{R}^N)$   endowed with the following combined norm:
	\begin{equation}\label{combspace}
		\|u\|_{N,r}=\left(\|\nabla u\|_N^N +\|\nabla u\|_r^N \right)^{1/N}.
	\end{equation}
	For $p\in (1,2)$, we recall the Sobolev critical exponent $p^*:=2p/(2-p)$ of the space $D^{1,p}(\mathbb R^2)$.
	Note that $E=H^{1}(\mathbb R^2)\cap D^{1,p}(\mathbb R^2)\subset E^{2,p}$.
	We summarize the embedding results from $E^{2,p}$ and $E$ to the Lebesgue space $L^r(\mathbb R^2)$. The proof of these embedding results can be completed using the Gagliardo--Nirenberg inequality and interpolation techniques.  We refer to  \cite{PW2018}*{Theorem 2.1} and \cite{CFFM2021}*{Proposition 2.1} for details and references. 
	\begin{lemma}\label{lem2.2}
		The embedding $E^{2,p}\hookrightarrow L^r(\mathbb{R}^2)$ is continuous for any $r\in [p^*,\infty)$. Additionally, the embedding $E^{2,p}_r:=\{u \in E : u \text{ is  radial}\}\hookrightarrow L^r(\mathbb{R}^2)$ is compact for any $r\in (p^*,\infty)$.
		Furthermore,  the embedding $E\hookrightarrow L^r(\mathbb R^N)$ is also continuous for any $r\in[2,\infty)$ and $E_r=\{u \in E : u \text{ is  radial}\}\hookrightarrow L^r(\mathbb{R}^2)$ is compact for any $r\in (2,\infty)$.
	\end{lemma}

	Following \cite{CFFM2021}*{Theorem 1.1},  we  introduce a version of the Trudinger--Moser inequality in $\mathbb{R}^2$. This advancement is significant because it provides bounds on integrals involving exponential-type functions, which in turn ensures the well-definedness of the energy functional. More precisely, we define the Young function $\varPhi_{\alpha, j_0}:\mathbb R\to \mathbb R$ by
	\begin{align*}
		\varPhi_{\alpha, j_0}(s):=e^{\alpha|s|^2}-\sum_{j=0}^{j_0-1}\frac{\alpha^j}{j!}|s|^{2j},
	\end{align*}
	where  $\alpha>0$ and   $j_0:=\inf\left\{j\in \mathbb{N} \st j\geq {p^*}/{2}\right\}$.

	\begin{lemma}\label{lem2.1}
		Suppose that $1<p<2$ and $u\in E^{2,p}$. Then for any $\alpha>0$, one has
		\begin{align*}
			\int_{\mathbb{R}^2} \varPhi_{\alpha,j_0}(u)\dif x<+\infty.
		\end{align*}
		Moreover, if $0<\alpha<4\pi$ and $\|u\|_{2,p}\leq 1$, then there exists a constant $C(\alpha)>0$ such that
		\begin{align*}
			\int_{\mathbb{R}^2} \varPhi_{\alpha,j_0}(u)\dif x\leq C(\alpha).
		\end{align*}
		In particular,  if $0<\alpha<4\pi$, $\|\nabla u\|_{2}\leq 1$, and $\|u\|_2\leq M<+\infty$ with $M>0$, then there exists a constant $C(M,\alpha)>0$ such that
		\begin{align*}
			\int_{\mathbb{R}^2} (e^{\alpha u^2}-1)\dif x\leq C(M,\alpha).
		\end{align*}
	\end{lemma}
	
	For $u \in E\setminus\{0\}$ and $s\in \R$,  we introduce the  fiber map defined by  
	\begin{equation*}
		s\star u=e^su(e^s\cdot).
	\end{equation*}
	As a result,  for all $s\in \R$ we have $\|s\star u\|_2=\|u\|_2$. Recalling the functional $I$ introduced in \eqref{enerfunct}, we now summarize its essential properties below.
	\begin{lemma}\label{lem2.4}
		Let $1<p<2$ and suppose that $f$ satisfies the assumptions $(f_1)$-$(f_2)$.   Then the functional $I(u)$ is of class $C^1$ on $E$. Moreover, for $u\in E\setminus\{0\}$, the composition function  satisfies
		$I(s\star u)\to 0^+$ as $s \to -\infty$. Additionally, if assumption $(f_4)$ holds, then $I(s\star u)\to -\infty$ as $s \to +\infty$.
	\end{lemma}
	\begin{proof} We first prove that $I(u)$ is well-defined. Since $u\in E$ and $p>1$, it suffices to verify that $F(u)\in L^1(\R^2)$. Indeed, by the assumptions $(f_1)$-$(f_2)$, for given $\alpha>\alpha_0$, $r\geq 1$ and $\varepsilon>0$,  there exists a positive constant $C_0:=C(\alpha,r,\varepsilon)$ such that, for any $s\in \mathbb{R}$,
		\begin{align} \label{f2.4}
			\abs{F(t)},\abs{f(t)t} \leq \varepsilon\abs{t}^{4}+C_0\abs{t}^{r}\varPhi_{\alpha,j_0}(t).
		\end{align}
		This, together with Lemmas   \ref{lem2.2} and  \ref{lem2.1}, implies that
		\begin{align*}
			\int_{\mathbb{R}^2}|F(u)|\dif x&\leq\varepsilon\int_{\mathbb{R}^2}|u|^4 \dif x+C_{0}\int_{\mathbb{R}^2}|u|^r\varPhi_{\alpha,j_0}(u)\dif x\\
			&\leq\varepsilon \|u\|^{4}_{4}+C_{0}\|u\|_{2r}^r \left(\int_{\mathbb{R}^2}\left(\varPhi_{\alpha,j_0}(u)\right)^{2}\dif x\right)^{1/2}\\
			&\leq\varepsilon \|u\|^{4}_{4}+C_{0}\|u\|^{r}_{2r}\left(\int_{\mathbb{R}^2}\varPhi_{2\alpha,j_0}(u)\dif x \right)^{1/2}<+\infty.
		\end{align*}
		Here we use the fact  from \cite{Yang12}*{Lemma 2.1} that for any $k\geq 1$,
		\begin{equation}\label{f8}
			\left(\varPhi_{\alpha,j_0}(t)\right)^k\leq \varPhi_{k\alpha,j_0}(t),\quad \text{for  all } t\in \mathbb{R}.
		\end{equation}
		Similarly, we can show that $\int_{\R^2}f(u)\phi\dif x$ is also well-defined for $u,\phi\in E$. Then by a standard argument, it follows that $I\in C^1(E,\mathbb{R})$, and for any $\phi \in E$,
		\begin{align*}
			I'(u)\phi=\int_{\mathbb{R}^2}\nabla u\nabla \phi+\abs{\nabla u}^{p-2}\nabla u \nabla \phi \dif x-\int_{\R^2}f(u)\phi\dif x.
		\end{align*}

		For $u\in E\setminus\{0\}$, we set $w_s(x):=({\alpha}/{\pi})^{1/2} (s\star u)$.
		Then it follows that
		\begin{align*}
			\|w_s\|^2_{2,p}
			=\frac{\alpha}{\pi}\Big(e^{2s}\|\nabla u\|^2_2+e^{4s\left(1-1/p\right)}\|\nabla u\|^2_p\Big)\to 0^+, \quad as\  s\to-\infty.
		\end{align*}
		Note that $\varPhi_{\alpha,j_0}(kt)=\varPhi_{\alpha k^2,j_0}(t)$ for  $t\in\mathbb{R}$ and $k>0$. By employing Lemma $\ref{lem2.1}$, we can find $\hat{s}<0$ small enough such that, for all $s\leq \hat{s}$,
		\begin{align*}
			\int_{\mathbb{R}^2}\varPhi_{2\alpha,j_0}(s\star u)\dif x 
			=\int_{\mathbb{R}^2}\varPhi_{2\pi,j_0}(w_s)\dif x\leq C^2_1
		\end{align*}
		for some $C_1>0$. Therefore, we have for $s\leq\hat{s}$ that,
		\begin{align}\label{eq2.4}
				\int_{\mathbb{R}^2}|F(s\star u)|\dif x
				&\leq\varepsilon \|s\star u\|^{4}_{4}+C_{0}\|s\star u\|^{r}_{2r}\left(\int_{\mathbb{R}^2}\varPhi_{2\alpha,j_0}(s\star u)\dif x \right)^{1/2}\nonumber\\
				&\leq \varepsilon e^{2s}\|u\|^4_4+C_1C_{0}e^{(r-1)s}\|u\|^r_{2r}\nonumber\\
				&\leq\varepsilon C_2 e^{2s}\|u\|^4_E+C_3C_1 C_{0}e^{(r-1)s}\|u\|_E^r.
		\end{align}
		Here the constants $C_2>0$ and $C_3 > 0$ arise from the continuous embedding inequalities $E\hookrightarrow L^{\tau}(\mathbb{R}^2)$, corresponding to
		$\tau = 4$ and $\tau = 2r$, respectively. We thus obtain for $s\leq\hat{s}$ that,
		\begin{align*}
			I(s\star u)=&\frac{e^{2s}}{2}\int_{\mathbb{R}^2}|\nabla u|^2\dif x+\frac{e^{2(p-1)s}}{p}\int_{\mathbb{R}^2}|\nabla u|^p\dif x-\int_{\mathbb{R}^2}F(s\star u)\dif x\\
			\geq&\frac{e^{2s}}{2}\|\nabla u\|^2_2+\frac{e^{2(p-1)s}}{p}\|\nabla u\|^p_p
			-\varepsilon C_2e^{2s}\|u\|_E^4-C_3C_1C_{0}e^{(r-1)s}\|u\|_E^r.
		\end{align*}
		Fix $r \geq 4>2p>2$,
		then there exists $s_0\leq\hat{s}$ such that $I(s\star u)\geq 0$ for all $s\leq s_0$.\ Moreover,  in view of \eqref{eq2.4}, we obtain that
		\begin{align*}
			\int_{\mathbb{R}^2}F(s\star u)\dif x\to 0, \ as\ s\to -\infty.
		\end{align*}
		It then follows  that $I(s\star u) \to 0^+\ as\ s\to-\infty$.
		
		Regarding $(f_1)$ and $(f_2)$,  for $r>4$, there exist $\kappa_1\in(0,1]$ and $C_4,C_5>0$ such that
		\begin{equation*}|F(t)|\leq C_4|t|^2,\quad \text{ if }  |t|\leq \kappa_1,
		\end{equation*}
		and
		\begin{equation*}	|F(t)|\geq C_5|t|^{r},\quad \text{ if }\  |t|\geq \kappa_1.
		\end{equation*}
		It then follows that,
		\begin{align*}
			\int_{\mathbb{R}^2}F(u)\dif x&=\int_{\{|u(x)|\leq \kappa_1\}}F(u)\dif x+\int_{\{|u(x)|\geq \kappa_1\}}F(u)\dif x\\
			&\geq-C_4\int_{\{|u(x)|\leq \kappa_1\}}|u|^2\dif x+C_5\int_{\{|u(x)|\geq \kappa_1\}}|u|^r \dif x\\
			&=- \int_{\{|u(x)|\leq \kappa_1\}}\left(C_4|u|^2+C_5|u|^{r}\right)\dif x+C_5\int_{\mathbb{R}^2}|u|^r \dif x\\
			&\geq-(C_4+C_5)\int_{\mathbb{R}^2}|u|^2\dif x+C_5\int_{\mathbb{R}^2}|u|^r\dif x.
		\end{align*}
		This leads us that
		\begin{align*}
			I(s\star u)\leq\frac{1}{2}e^{2s}\|\nabla u\|^2_2+\frac{1}{p}e^{2(p-1)s}\|\nabla u\|^p_p+(C_4+C_5)\|u\|_2^2-C_5e^{(r -2)s}\|u\|^{r}_{r}.
		\end{align*}
		Since  $r>4$, we conclude that $I(s\star u)\to -\infty$ once $s\to +\infty$.
	\end{proof}

	{\bf The proof of Theorem \ref{phothm}}.
		We claim that $u\in L_{\rm loc}^{\infty}(\mathbb{R}^2)$.
		We proceed with the proof, deferring the verification of this claim to a later stage.
		Since $u\in E$, we conclude that
		$u\in W^{1,2}_{\text{loc}}(\mathbb{R}^2)\cap L^{\infty}_{\text{loc}}(\mathbb{R}^2)$.\ Set
		\begin{align*}
			g(x):=-\lambda u(x)+ f(u(x)).
		\end{align*}
		Given that $f \in C^1(\mathbb{R}, \mathbb{R})$, it follows that $f(u(\cdot))\in  L^{\infty}_{\text{loc}}(\mathbb{R}^2)$. Consequently, we also have $g\in L^{\infty}_{\text{loc}}(\mathbb{R}^2)$. Note that $u\in E$ is a weak solution of  \eqref{2pnls}. We are permitted to utilize the arguments of \cite{HeLi08}*{Theorem 1} to conclude that
		$u\in C^{1,\gamma}_{\text{loc}}(\mathbb{R}^2)$ for some $\gamma \in (0,1)$.
		
		We now prove that $u$ satisfies the Pohozaev identity $P(u)=0$. For $\vec{s}=(s_1,s_2)\in \mathbb R^2$, we set \begin{equation*}
			\mathcal{L}\left(
		\vec{s}\right):=\frac{1}{p} \abs{\vec{s}}^p+\frac{1}{2}\abs{\vec{s}}^2.
		\end{equation*}
		Note that $p>1$. The function $\mathcal{L}$ is of class $C^1$ on $\mathbb R^2$ and is strictly convex respect to $\vec{s}$. Given that $u\in C^{1,\gamma}_{\rm loc}(\mathbb R^2)$. We observe that $u$ is a locally Lipschitz solution of
		\begin{equation*} -\text{div}(\nabla_s \mathcal{L}(\nabla u))=g(u(x)).
		\end{equation*}
		We then conclude from \cite{DMS03}*{Lemma 1} that  for every $\phi\in C^1_{\text{c}}(\R^2,\R^2)$,
		\begin{equation}\label{ph1}
			\sum_{i,j=1}^{2}\int_{\mathbb{R}^2}D_i\phi_jD_{s_i}\mathcal{L}\left(\nabla u\right)D_ju\dif x-\int_{\mathbb{R}^2}(\text{div}\phi)\mathcal{L}\left(\nabla u\right)\dif x=\int_{\mathbb{R}^2}\left(\phi\cdot\nabla u\right)g(u)\dif x.
		\end{equation}
	
		Let us consider a radial cut-off function $\varphi\in C^1_{\text{c}}(\R^2)$ such that $\varphi=1$ on $B_1$ and $\abs{\varphi'(r)r}\leq C$.   By substituting the function $\phi $ in equation \eqref{ph1}  with $\phi_R\in C^1_{\text{c}}(\R^2,\R^2)$,  where $\phi_{R}(x)=\varphi(Rx)x$  for $R\in (0,\infty)$ and $x\in \R^2$,  we deduce
		\begin{align*}
			\sum_{i,j=1}^{2}\int_{\mathbb{R}^2}D_i(\phi_R)_jD_{s_i}\mathcal{L}\left(\nabla u\right)D_ju\dif x-\int_{\mathbb{R}^2}(\text{div}\phi_R)\mathcal{L}\left(\nabla u\right)\dif x=\int_{\mathbb{R}^2}\left(\phi_R\cdot\nabla u\right)g(u)\dif x.
		\end{align*}
		We calculate the right-hand side by performing integration by parts for every $R>0$,
		\begin{align*}
			\int_{\mathbb{R}^2}\left(\phi_R\cdot\nabla u\right)g(u)\dif x
			&=\int_{\mathbb{R}^2} -\frac{\lambda}{2} \varphi(Rx)x\cdot \nabla\left(| u|^2\right)+ \varphi(Rx)x\cdot \nabla\left(F(u)\right)\dif x\\
			&=\int_{\mathbb{R}^2} \left(2\varphi(Rx)+Rx\nabla \varphi(Rx)\right)\Big(\frac{\lambda}{2}\abs{u}^2-F(u)\Big) \dif x.
		\end{align*}
		Given that $u\in E$, and applying Lebesgue's dominated convergence theorem, we obtain
		\begin{equation*}
			\lim_{R\to 0}\int_{\mathbb{R}^2}\left(\phi_R\cdot\nabla u\right)g(u)\dif x=\lambda\| u\|^2_2-2\int_{\R^2}F(u)\dif x.
			\end{equation*}
		Similarly, by a straightforward calculation, we arrive that
		\begin{align*}
			\lim_{R\to 0}\int_{\mathbb{R}^2}(\text{div}\phi_R)\mathcal{L}\left(\nabla u\right)\dif x
			=&\lim_{R\to 0}\int_{\R^2}\left(2\varphi(Rx)+R\nabla \varphi(Rx)\cdot x\right)\mathcal{L}(\nabla u)\dif x\\
			=&\frac{2}{p}\int_{\mathbb{R}^2}|\nabla u|^p\dif x+\int_{\mathbb{R}^2}|\nabla u|^2\dif x.
		\end{align*}
		On the other hand, it is observed that
		\begin{align*}
			& \sum_{i,j=1}^{2}\int_{\mathbb{R}^2}D_i(\phi_R)_jD_{s_i}\mathcal{L}\left(\nabla u\right)D_ju\dif x\\
			=& \int_{\R^2}\varphi(Rx)\left(|\nabla u|^p+|\nabla u|^2\right)\dif x + R\int_{\mathbb R^2} \left(\nabla\varphi(Rx)\cdot\nabla_s\mathcal{L}(\nabla u)\right) (x\cdot \nabla u) \dif x,
		\end{align*}
		and
		\begin{equation*}
			\abs{\left(\nabla\varphi(Rx)\cdot\nabla_s\mathcal{L}(\nabla u)\right) (Rx\cdot \nabla u)}\leq \abs{\nabla\varphi(Rx)}\abs{Rx}\mathcal{L}(\nabla u)\leq C\mathcal{L}(\nabla u).
		\end{equation*}
		These inequalities allow us to once again apply Lebesgue's dominated convergence theorem, leading to the conclusion that
		\begin{align*}
			\lim_{R\to0} \sum_{i,j=1}^{2}\int_{\mathbb{R}^2}D_i(\phi_R)_jD_{s_i}\mathcal{L}\left(\nabla u\right)D_ju\dif x
			=\int_{\mathbb{R}^2}|\nabla u|^p+\abs{\nabla u}^2\dif x.
		\end{align*}
		We thus conclude that
		\begin{equation}\label{ig1}
			\frac{p-2}{p}\|\nabla u\|^p_p=\lambda\|u\|^2_2-2\int_{\R^2}F(u)\dif x.
		\end{equation}
		Multiplying $\eqref{2pnls}$ by $u$, we derive the Nehari-type identity:
		\begin{equation}\label{ig2}
			\|\nabla u\|^2_2+\|\nabla u\|^p_p+\lambda\|u\|_2^2=\int_{\mathbb{R}^2} f(u)u\dif x.
		\end{equation}
		Combining \eqref{ig1} with \eqref{ig2}, we ultimately find that  $ P(u)=0$.

		\indent In what follows, we return to prove the claim that $u\in L^\infty_{\text{loc}}(\mathbb R^2)$. Let $\Omega\neq \emptyset$ be any compact set of $\R^2$.
		For each $m\in\mathbb{N}$ and $\beta>1$, we define the sets  
		\begin{equation*}
			X_{m}:=\{x\in\Omega: \abs{u(x)}^{\beta-1}\leq m\} \text{ and } \, Y_m:=\Omega\setminus X_m.
			\end{equation*}
		Let
		\begin{align*}
			u_m(x):=
			\begin{cases}
				u\abs{u}^{p(\beta-1)},\qquad & x\in X_m,\\
				m^{p}u, \qquad & x\in Y_m,
			\end{cases}
			\quad 	\text{and } \quad
			v_m(x):=
			\begin{cases}
				u|u|^{\beta-1},\qquad &x\in X_m,\\
				mu,\qquad &x\in Y_m.
			\end{cases}
		\end{align*}
		By  defining $u_m(x)=0$ and $v_m(x)=0$ for all $x\not\in \Omega$, we extend the functions $u_m$ and $v_m$ to the whole space, consequently we have  $u_m, v_m\in E$ for any $m\in\mathbb{N}$ and for $x\in \Omega$,
		\begin{align*}
			\nabla u_m(x)=
			\begin{cases}
				(p\beta-p+1)\abs{u}^{p(\beta-1)}\nabla u,\, & x\in  X_m,\\
				m^{p}\nabla u,\,  & x\in   Y_m,
			\end{cases}
			\,   
			\nabla v_m(x)=
			\begin{cases}
				\beta \abs{u}^{\beta-1} \nabla u,\, &x\in X_m,\\
				m\nabla u,\, &x\in Y_m.
			\end{cases}
		\end{align*}
Since $u$
is a weak solution to Problem \eqref{2pnls}, testing it against $u_m$ gives that
		\begin{equation}\label{f1}
			\int_{\Omega}\nabla u \nabla u_m \dif x+\int_{\Omega}\abs{\nabla u}^{p-2}\nabla u \nabla {u}_m\dif x+\lambda\int_{\Omega} u {u}_m\dif x=\int_{\Omega} f(u) {u}_m\dif x.
		\end{equation}
		By a straightforward calculation, we obtain
		\begin{align}\label{pterm}
			\int_{\Omega}|\nabla v_m|^p\dif x=\beta^p\int_{X_m}\abs{u}^{p(\beta-1)}\abs{\nabla u}^p\dif x+m^p\int_{Y_m}\abs{\nabla u}^p\dif x,
		\end{align}	
		and
		\begin{equation}\label{f2}
			\int_{\Omega}\abs{\nabla u}^{p-2}\nabla u \nabla u_m\dif x=(p\beta -p+1)\int_{X_m}\abs{  u}^{p(\beta-1)}\abs{\nabla u}^p\dif x+m^p\int_{Y_m}\abs{\nabla u}^p\dif x.
		\end{equation}
Subtracting \eqref{f2} from \eqref{pterm}, we deduce that
\begin{equation}\label{f3}
\int_{\Omega}|\nabla v_m|^p\dif x-\int_{\Omega}\abs{\nabla u}^{p-2}\nabla u \nabla u_m\dif x=(\beta^p-p\beta+p-1)\int_{X_m}\abs{u}^{p(\beta-1)}\abs{\nabla u}^p\dif x.
\end{equation}
Note that $\beta>1$ and $1<p<2$.  We see that
\begin{equation*}
\beta^p-p\beta+p-1>0  \quad \text{ and  } \quad p\beta-p+1>1.
\end{equation*}
 Therefore, by combining \eqref{f2} and \eqref{f3}, we conclude
\begin{equation}\label{f5}
\int_{\Omega}|\nabla v_m|^p\dif x\leq\frac{\beta^p}{p\beta-p+1}\int_{\Omega}\abs{\nabla u}^{p-2}\nabla u \nabla u_m\dif x\leq \beta^p\int_{\Omega}\abs{\nabla u}^{p-2}\nabla u \nabla u_m\dif x.
\end{equation}
By virtue of \eqref{f1} and \eqref{f5}, we derive
\begin{align*}
\int_{\Omega}|\nabla v_m|^p\dif x&\leq\beta^p\int_{\Omega}\abs{\nabla u}^{p-2}\nabla u \nabla u_m\dif x\\
&=-\beta^p\int_{\Omega}\nabla u \nabla u_m \dif x-\beta^p\lambda\int_{\Omega}uu_m\dif x+\beta^p\int_{\Omega}f(u)u_m\dif x.
\end{align*}
Noticing that
\begin{align*}
\int_{\Omega}\nabla u \nabla u_m \dif x
=  (p\beta-p+1)\int_{X_m}\abs{u}^{p(\beta-1)}\abs{\nabla u}^2\dif x+m^p\int_{Y_m}\abs{\nabla u}^2\dif x\geq 0,
\end{align*}
we thus obtain
\begin{equation}\label{f6}
\int_{\Omega}|\nabla v_m|^p\dif x\leq\beta^p\abs{\lambda}\int_{\Omega}\abs{uu_m} \dif x+\beta^p\int_{\Omega}f(u)u_m\dif x.
\end{equation}
\indent	We shall estimate the terms on the right-hand side of \eqref{f6} separately. On  one hand, regarding that $\abs{u_m}\leq \abs{u}^{p\beta-p+1}$ in $\Omega$ and using the H\"{o}lder inequality, we conclude
\begin{align}\label{f7}
\int_{\Omega}\abs{uu_m}\dif x\leq \int_{\Omega}\abs{u}^{p\beta-p+2}\dif x
&\leq \abs{\Omega}^{p/4}\|u\|_2^{2-p}\|u\|^{p\beta}_{L^{4\beta}(\Omega)}:=C_1 \|u\|^{p\beta}_{L^{4\beta}(\Omega)},
\end{align}
where $C_1=C_1(u,\Omega)>0$ and $\abs{\Omega}$ denotes the measure of $\Omega$. On the other hand, for any $\varepsilon>0$ and  $\alpha>\alpha_0$, there exists $C_{\varepsilon}>0$ such that for any $t\in\mathbb{R}$, 
\begin{align}\label{keyesti}
\abs{f(t)}\leq \varepsilon\abs{t}+C_{\varepsilon}\abs{t}\varPhi_{\alpha,j_0}(t).
\end{align}
This, together with the fact that $\abs{u_m}\leq \abs{u}^{p\beta-p+1}$, implies that
\begin{align*}
\int_{\Omega}f(u)u_m\dif x
&\leq \int_{\Omega}\big(\varepsilon\abs{u}^{p\beta-p+2}+C_{\varepsilon}\abs{u}^{p\beta-p+2}\varPhi_{\alpha,j_0}(u)\big)\dif x.
\end{align*}
In view of \eqref{f7}, we deduce
\begin{align*}
\int_{\Omega}\abs{u}^{p\beta-p+2}\dif x\leq C_1\|u\|^{p\beta}_{L^{4\beta}(\Omega)}. 
\end{align*}
We then deduce from Lemmas \ref{lem2.2}, \ref{lem2.1}  and \eqref{f8} that
\begin{align*}
\int_{\Omega}&\abs{u}^{p\beta-p+2}\varPhi_{\alpha,j_0}(u)\dif x
\leq \|{u}\|_{L^{2\beta}(\Omega)}^{p\beta}
\left(\int_{\Omega}\abs{u}^{2}\varPhi_{\alpha,j_0}^{2/(2-p)}(u)\dif x\right)^{1-p/2}\\
&\leq \abs{\Omega}^{p/4}\|u\|_{L^{4\beta}(\Omega)}^{p\beta}\|u\|_{L^4(\Omega)}^{2-p}\left(\int_{\Omega}\varPhi_{\frac{4\alpha}{2-p},j_0}(u)\dif x\right)^{\frac{2-p}{4}}:=C_2\|u\|^{p\beta}_{L^{4\beta}(\Omega)},
\end{align*}
where $C_2=C_2(u,\Omega)>0$. We therefrom conclude that
\begin{equation}\label{f9}
\int_{\Omega}f(u)u_m\dif x\leq \left(\varepsilon C_1+C_{\varepsilon}C_2\right)\|u\|^{p\beta}_{L^{4\beta}(\Omega)}:= C_3\|u\|^{p\beta}_{L^{4\beta}(\Omega)}.
\end{equation}
Consequently, by combining \eqref{f6},\eqref{f7} and \eqref{f9}, we obtain
\begin{equation}\label{f10}
\int_{\Omega}|\nabla v_m|^p\dif x\leq(C_1|\lambda|+C_3)\beta^p\|u\|^{p\beta}_{L^{4\beta}(\Omega)}:=C_4\beta^p \|u\|^{p\beta}_{L^{4\beta}(\Omega)}.
\end{equation}
Similarly, by replacing $p$ with $2$ in the definition of $u_m$, denoted by $w_m$ for instance,  we then proceed the same arguments as above and derive
\begin{align*}
\int_{\Omega}|\nabla v_m|^2 \dif x =\beta^2\int_{X_m}\abs{u}^{2(\beta-1)}\abs{\nabla u}^2 \dif x+ m^2 \int_{Y_m} \abs{\nabla u}^2 \dif x.
\end{align*}
Note that
\begin{align*}
\int_{\Omega}\nabla u\nabla w_m \dif x
=  (2\beta-1)\int_{X_m}\abs{u}^{2(\beta-1)}\abs{\nabla u}^2\dif x+m^2\int_{Y_m}\abs{\nabla u}^2\dif x.
\end{align*}
We then obtain that
\begin{align*}
\int_{\Omega}|\nabla v_m|^2\dif x-\int_{\Omega} \nabla u \nabla w_m\dif x=(\beta^2-2\beta+1)\int_{X_m}\abs{u}^{2(\beta-1)}\abs{\nabla u}^2\dif x.
\end{align*}
and
\begin{equation}\label{wterm}
\int_{\Omega} \abs{\nabla v_m}^2\leq \frac{\beta^2}{2\beta-1}\int_{\Omega}\nabla u\nabla w_m \dif x\leq \beta^2\int_{\Omega}\nabla u \nabla w_m.
\end{equation}
Again by testing the equation \eqref{2pnls} against $w_m$, we deduce
\begin{align*}
\int_{\Omega}\nabla u \nabla w_m \dif x+\int_{\Omega}\abs{\nabla u}^{p-2}\nabla u \nabla w_m\dif x+\lambda\int_{\Omega} u w_m\dif x=\int_{\Omega} f(u) w_m\dif x.
\end{align*}
This, together with \eqref{wterm}  implies that
\begin{align}\label{fg11}
\int_{\Omega}\abs{\nabla v_m}^2\dif x&\leq\beta^2\int_{\Omega}\left(-\abs{\nabla u}^{p-2}\nabla u\nabla w_m-\lambda u w_m+f(u)w_m\right)\dif x&\nonumber\\
& \leq \beta^2 \abs{\lambda}\int_{\Omega}\abs{uw_m}\dif x +\beta^2\int_{\Omega}\abs{f(u)w_m}\dif x&\nonumber\\
&\leq \beta^2 (\abs{\lambda}+\varepsilon)\int_{\Omega}\abs{u}^{2\beta}\dif x +\beta^2C_{\varepsilon}\int_{\Omega}\abs{u}^{2\beta}\varPhi_{\alpha,j_0}(u)\dif x.
\end{align}
We observe from the H\"older inequality that
\begin{equation}\label{fg111}
\int_{\Omega}\abs{u}^{2\beta}\dif x\leq
\abs{\Omega}^{1/2}\|u\|^{2\beta}_{L^{4\beta}(\Omega)}:= C_5\|u\|^{2\beta}_{L^{4\beta}(\Omega)}.
\end{equation}
It then follows from  Lemma \ref{lem2.1} that
\begin{align}\label{fg12}
\int_{\Omega}\abs{u}^{2\beta}\Phi_{\alpha,j_0}(u)\dif x\leq  \|u\|_{L^{4\beta}(\Omega)}^{2\beta} \left(\int_{\Omega}\Phi_{2\alpha,j_0}(u)\dif x\right)^{1/2}:=C_6\|u\|^{2\beta}_{L^{4\beta}(\Omega)}.
\end{align}
By combining \eqref{fg111} and \eqref{fg12}, we can deduce from \eqref{fg11} that
\begin{align}\label{fg13}
\int_{\Omega}|\nabla v_m|^2\dif x \leq C_5(\abs{\lambda}+\varepsilon)\beta^2\|u\|_{L^{4\beta}(\Omega)}^{2\beta}+ C_{\varepsilon}C_6\beta^2\|u\|_{L^{4\beta}(\Omega)}^{2\beta}:=C_7\beta^2\|u\|_{L^{4\beta}(\Omega)}^{2\beta}.
\end{align}
From \eqref{f10} and \eqref{fg13}, we obtain that
\begin{align*}
\|v_m\|^2_{2,p}&=\int_{\Omega}|\nabla v_m|^2\dif x+\left(\int_{\Omega}|\nabla v_m|^p\dif x\right)^{2/p}\\
&\leq  C_7\beta^2 \|u\|^{2\beta}_{L^{4\beta}(\Omega)}+ C_4^{2/p}\beta^2\|u\|^{2\beta}_{L^{4\beta}(\Omega)}:=C_8\beta^2\|u\|^{2\beta}_{L^{4\beta}(\Omega)}.
\end{align*}
According to Lemma \ref{lem2.2}, for $p_1 >\max\{4,p^*\}$,  we arrive at  by noting $|v_m|=|u|^{\beta}$ in $X_m$ that,
\begin{align*}
\|u\|^{\beta}_{L^{p_1\beta}(X_m)}=	\left(\int_{X_m}|v_m|^{p_1}\dif x\right)^{1/p_1}\leq C_9 \|v_m\|_{2,p} \leq  C_{10}\beta\|u\|^{\beta}_{L^{4\beta}(\Omega)}.
\end{align*}
Letting $m\to\infty$ and using the Fatou lemma, we obtain
\begin{equation}\label{f12}
\|u\|_{L^{p_1\beta}(\Omega)}\leq \beta^{1/\beta}C_{10}^{1/\beta} \|u\|_{L^{4\beta}(\Omega)},\qquad \text{for} \ \beta>1,\ p_1>\max\{4,p^*\}.
\end{equation}
\indent Now we start the Moser's iteration procedure.
By setting $\beta=\delta=p_1/4>1$, we obtain by virtue of \eqref{f12} that
\begin{equation*}
\|u\|_{L^{p_1\delta}(\Omega)}\leq \delta^{1/\delta}C_{10}^{1/\delta} \|u\|_{L^{p_1}(\Omega)}.
\end{equation*}
By taking $\beta=\delta^2$, we deduce
\begin{align*}
\|u\|_{L^{p_1\delta^2}(\Omega)}\leq \delta^{1/\delta+2/\delta^2}C_{10
}^{1/\delta+1/\delta^2} \|u\|_{L^{p_1}(\Omega)}.
\end{align*}
Repeating this process, we generally obtain
\begin{equation*}
\|u\|_{L^{p_1\delta^j}(\Omega)}\leq \delta^{\sum_{j=1}^{\infty}j/\delta^j}C_{10}^{\sum_{j=1}^{\infty}1/\delta^j} \|u\|_{L^{p_1}(\Omega)}.
\end{equation*}
Since $\delta>1$, there exists a constant $C_{11}:=C_{11}(u,\Omega)>0$
such that
\begin{align*}
\|u\|_{L^{p_1\delta^j}(\Omega)}\leq  C_{11}\|u\|_{L^{p_1}(\Omega)}.
\end{align*}
By  letting $j\to\infty$ and  noting Lemma \ref{lem2.2}, there exists $C_{12}>0$ such that
\begin{align*}
\|u\|_{L^{\infty}(\Omega)}\leq C_{12}\|u\|_{2,p}<+\infty.
\end{align*}
Then the claim follows. \qed

\section{Asymptotic behavior and characterization of  $\varUpsilon_m$ }
 
\label{sect3}

\indent
Recall  that $H(t)= f(t)t-2F(t)$. We define $K(t):=f(t)t-4F(t)$.
The first lemma summarizes some useful properties of $H$ and $K$.
\begin{lemma}\label{lem2.6}
Assume that $f$ satisfies $(f_1)$-$(f_3)$. Then,
\begin{enumerate}
	\item[{\rm (i)}] $\dfrac{H(t)}{t^4}$ is nonincreasing on $(-\infty, 0)$ and nondecreasing on $(0,+\infty)$;
	\item[{\rm (ii)}]$\dfrac{K(t)}{t^2}$ is nonincreasing on $(-\infty, 0)$ and nondecreasing on $(0,+\infty)$;
	\item[{\rm (iii)}]for all $t\in \mathbb R$, the functions $K(t)$, $H(t)$ and $F(t)\geq 0$.
\end{enumerate}
\end{lemma}
\begin{proof} For $t\in \R\setminus\{0\}$, we define
\begin{equation*}
	\psi(t):=\frac{H(t)}{t^4}\ \text{ and } \ \vartheta(t):=\frac{K(t)}{t^2}.
\end{equation*}
Since $f\in C^1(\mathbb{R},\mathbb{R})$ and $K(t)=H(t)-2F(t)$, it is easy to verify that 
\begin{equation*}
	\psi^{\prime}(t)=\frac{h(t) t-4 H(t)}{t^5}  \   \text{ and }     \  \vartheta^{\prime}(t)=\frac{h(t) t-4 H(t)}{t^3}.
\end{equation*}
Consequently, conclusions (i) and (ii) follow  immediately from $(f_3)$.
For the item (\romannumeral 3), we define $\tilde{\vartheta}(t)=\vartheta(t)$ for $t\neq 0$ and $\tilde{\vartheta}(0)=0$.
By virtue of $(f_1)$-$(f_3)$, it follows that $\tilde{\vartheta}$ is continuous on $\mathbb{R}$. In light of this and conclusion  (\romannumeral 2), we deduce that $\tilde{\vartheta}(t)\geq0$ for all $t\in \mathbb{R}$, which implies
$K(t)\geq 0$. Similarly, since $\psi(t)\geq 0$ for $t\neq 0$, by an argument analogous to that for $\tilde{\vartheta}$,
we conclude $H(t)\geq 0$ for all $t\in \mathbb R$.  This consequently ensures that  $F(t)\geq 0$ for all $t\in \mathbb R$.
\end{proof}

\begin{lemma}\label{lem2.7}
Assume that $f$ satisfies $(f_1)$-$(f_3)$. Then the following statements hold:
\begin{enumerate}
	\item[{\rm (i)}] for any $u \in \cB_m\setminus\{0\}$, there exists a unique $s_u\in \mathbb{R}$ such that  
	\begin{equation*}
		I\left(s_u\star u\right)=\max_{s\in \mathbb{R}}I\left(s\star u\right).
	\end{equation*}
	As a result,  $\cB_m\cap\mathcal{P}\neq\emptyset$.
	Moreover, if $P(u)\leq 0$, then $s_u\leq 0$.
	\item[{\rm (ii)}] the map $s_u: \mathcal{S}_m \to \R$ is continuous.
	\item[{\rm (iii)}]there exists $\tau_0>0$ such that
	\begin{equation*}\inf\limits_{u \in \mathcal{P}_m}\|u\|_{2,p}\geq \inf\limits_{u \in \mathcal{B}_m\cap\mathcal{P}}\|u\|_{2,p}\geq\tau_0.
		\end{equation*}
\end{enumerate}
\end{lemma}
\begin{proof} Let $u\in E\setminus\{0\}$, from Lemma $\ref{lem2.4}$, we observe that $\varphi_u(s):=I(s\star u)$ admits a global maximum at some $s_u\in \mathbb{R}$, so that we have $\varphi_u'(s_u)=0$. By direct calculations,
\begin{align*}
	&\varphi_u'(s) =e^{2s}\|\nabla u\|^2_2+\frac{2(p-1)}{p}e^{2(p-1)s}\|\nabla u\|^p_p+e^{-2s}\int_{\mathbb{R}^2}\left(2F(e^su)-f(e^su)e^su\right)\dif x\\
	&=\|\nabla \left(e^su(e^sx)\right)\|^2_2+\frac{2(p-1)}{p}\|\nabla \left(e^su(e^sx)\right)\|^p_p-\int_{\mathbb{R}^2}H\left(e^su(e^sx)\right)\dif x=P(s\star u).
\end{align*}
We then reach that $P\left(s_u\star u\right)=0$ and $I\left(s_u\star u\right)= \max_{s\in \mathbb{R}}I\left(s\star u\right)$.
To prove the uniqueness of $s_u$, we assume by contradiction that there exists $s_1,s_2\in \mathbb{R}$ such that $P\left(s_1\star u\right)=P\left(s_2\star u\right)=0$ and $s_1< s_2$. Therefore, we have
\begin{align*}
	0=&e^{-2s_1}P\left(s_1\star u\right)-e^{-2s_2}P(s_2\star u)\\
	=&\frac{2(p-1)}{p}\big(e^{2(p-2)s_1}-e^{2(p-2)s_2}\big)\|\nabla u\|^p_p+\int_{\mathbb{R}^2}\Big(\frac{H(e^{s_2}u)}{e^{4s_2}}-\frac{H(e^{s_1}u)}{e^{4s_1}}\Big)\dif x.
\end{align*}
This, together with Lemma $\ref{lem2.6}$ and the fact that $H(t)$ is even, implies that
\begin{align*}
	0<\frac{2(p-1)}{p} \big(e^{2(p-2)s_1}-e^{2(p-2)s_2}\big)\|\nabla u\|^p_p
	=&\int_{\mathbb{R}^2}\Big(\frac{H(e^{s_1}u)}{e^{4s_1}u^4}-\frac{H(e^{s_2}u)}{e^{4s_2}u^4}\Big)u^4\dif x\leq0,
\end{align*}
which is a contradiction. Hence, for any $u\in \mathcal{B}_m\setminus\{0\}$, there exists a unique $s_u\in \mathbb{R}$ such that $(s_u\star u)\in \mathcal{P}$. Note that $(s_u\star u)\in \mathcal{B}_m$. We then conclude that $\mathcal{B}_m\cap \mathcal{P}\neq \emptyset$. We suppose that $s_u>0$ if $P(u)\leq 0$. Note that $P(u)=P(0\star u)\leq 0$ and $P(s_u\star u)=0$. By setting $s_1=0$ and $s_2=s_u$ and repeating the similar argument as above, we deduce  a contradiction again:
\begin{align*}
	0<\frac{2(p-1)}{p} \big(e^{2(p-2)s_1}-e^{2(p-2)s_2}\big)\|\nabla u\|^p_p
	\leq &\int_{\mathbb{R}^2}\Big(\frac{H(e^{s_1}u)}{e^{4s_1}u^4}-\frac{H(e^{s_2}u)}{e^{4s_2}u^4}\Big)u^4\dif x\leq0.
\end{align*}

We now turn to the continuity of $s_u$.  Let $\{u_n\}\subset \mathcal{S}_m$, satisfy $u_n\to  u$ strongly in $E$, so that we have $u\in \mathcal{S}_m$ and $\{u_n\}$ is bounded $E$.  Setting $s_n:=s_{u_n}$ for any $n\geq 1$, up to a subsequence, it  suffices to prove $s_n\to s_u$ as $n\to \infty $. We first prove that $\{s_n\}$ is bounded. In fact, if $s_n\to +\infty$, in view of Lemma $\ref{lem2.4}$ , we reach a contradiction that
$0<I(s_n\star u_n)\to -\infty$. Therefore, the sequence $\{s_n\}$ is bounded from above. Since $u_n\to u$ strongly in $E$, by noting the fact that $I\in C^1$ on $E$, we deduce that
\begin{align*}
	I(s_u\star u_n)=I(s_u\star u)+o_n(1).
\end{align*}
Note that $s_n\star u_n\subset \mathcal{P}_m$ for any $n\geq 1$. We then observe that
$I(s_n\star u_n)\geq I(s_u\star u_n)$, we thus conclude
\begin{align}\label{eql1}
	\liminf_{n\to\infty}I(s_n\star u_n)\geq I(s_u\star u)>0.
\end{align}
On the other hand, if $s_n\to -\infty$ as $n\to\infty$, by considering $F(t)\geq 0$, we obtain that
\begin{align*}
	I(s_n\star u_n)&=\frac{1}{2}e^{2s_n}\|\nabla u_n\|^2_2+\frac{1}{p}e^{2(p-1)s_n}\|\nabla u_n\|^p_p-\int_{\R^2}F(s_n\star u_n)\dif x&\\
	&\leq\frac{1}{2}e^{2s_n}\|\nabla u_n\|^2_2+\frac{1}{p}e^{2(p-1)s_n}\|\nabla u_n\|^p_p\to 0^+,
\end{align*}
which contradicts to $\eqref{eql1}$.
Without loss of generality, up to a subsequence, we may assume that $s_n\to \tilde{s}$, for some $\tilde{s}\in \R$.
By observing the fact that $s_n\star u_n \to \tilde{s}\star u$ strongly in $E$ as $n\to\infty$, we infer that $P(\tilde{s}
\star u)=\lim_{n\to\infty}P(s_n\star u_n)=0$. From the uniqueness of $s_u$, one then obtain $\tilde{s}=s_u$. This completes the continuity of $s_u$ with respect to $u\in \mathcal{S}_m$.

Finally, we complete the proof of the last conclusion by a contradiction argument. Since $\mathcal{P}_m\subset \cB_m\cap \mathcal{P}$, it suffices to show that the latter infimum value is strictly positive.
Suppose that there exists a sequence $\{u_n\}\subset \cB_m\cap\mathcal{P}$ such that $\|u_n\|_{2,p}\to 0 $ as $n\to \infty$.  Then, for $\alpha>\alpha_0$,  by noting Lemma \ref{lem2.1}, there exists $C_0>0$ such that
\begin{align}\label{upperbound}
	\int_{\mathbb{R}^2} \varPhi_{2\alpha,j_0}(u_n)\dif x=\int_{\mathbb{R}^2} \varPhi_{2\pi,j_0}(\sqrt{{\alpha}/{\pi}}u_n)\dif x\leq C^2_0.
\end{align}
By using similar arguments as \eqref{f2.4} in Lemma \ref{lem2.4}, and by considering Lemma $\ref{lem2.2}$, for any $\varepsilon>0$,  $r\geq p^*>2$, there exists $C_\varepsilon>0$ such that
\begin{align*}
	&\int_{\R^2}F(u)\dif x\leq \int_{\mathbb{R}^2} f(u_n)u_n \dif x &\nonumber\\
	&\leq\varepsilon \|u_n\|^{4}_{4}+C_{\varepsilon}\|u_n\|_{2r}^{r}\left(\int_{\mathbb{R}^2}\varPhi_{2\alpha,j_0}(u_n)\dif x \right)^{1/2}&\nonumber\\
	&\leq\varepsilon  \|u_n\|^{4}_{4}+C_{\varepsilon}C_0C_1\|u_n\|^{r}_{2,p}.
\end{align*}
Here the constant $C_1>0$ comes from the embedding $E_{2,p}\hookrightarrow L^{2r}(\R^2)$.

We first deal with the case that $4\geq p^*$.  By Lemma \ref{lem2.1}, there exists a constant   $C_2>0$,  independent of $n$, such that
$\|u_n\|_4^4\leq C_2\|u_n\|_{2,p}^4$. Fixing $\varepsilon={1}/{(3C_2)}$ and $C_{\varepsilon}$, we obtain for largely $n$,
\begin{align*}
	\|u_n\|^2_{2,p}=\|\nabla u_n\|^2_2+\|\nabla u_n\|^2_p
	&\leq\|\nabla u_n\|^2_2+2(1-1/p)\|\nabla u_n\|^p_p&\nonumber\\
	&=\int_{\mathbb{R}^2}\left(f(u_n)u_n-2F(u_n)\right) \dif x&\nonumber\\
	&\leq  \|u_n\|^{4}_{2,p}+3C_{\varepsilon}C_0 C_1\|u_n\|^{r}_{2,p}.
\end{align*}
This gives a contradiction since $r>2$. For $p^*>4$, we apply the interpolation inequality for any $u\in L^2(\R^2)\cap L^{p^*}(\R^2)$,
\begin{equation*}
	\|u\|_4\leq \|u\|_2^{1-\kappa}\|u\|_{p^*}^\kappa,  \text{ where } 4\kappa=\frac{p}{p-1},
\end{equation*}
to deduce that
\begin{align*}
	\int_{\R^2}F(u)\dif x\leq \int_{\mathbb{R}^2}f(u_n)u_n\dif x
	&\leq\varepsilon  \|u_n\|^{4}_{4}+C_{\varepsilon}C_0 C_1\|u_n\|^{r}_{2,p}&\nonumber\\
	&\leq \varepsilon \|u_n\|^{4(1-\kappa)}_2\|u_n\|_{p^*}^{4\kappa}+C_{\varepsilon}C_0 C_1\|u_n\|^{r}_{2,p}.
\end{align*}
By combining the embedding: $E_{2,p}\hookrightarrow L^{p^*}(\R^2)$,  we  also obtain for largely $n$,
\begin{align}\label{lowerbound}
	\|u_n\|^2_{2,p}
	&\leq\|\nabla u_n\|^2_2+2(1-1/p)\|\nabla u_n\|^p_p&\nonumber\\
	&=\int_{\mathbb{R}^2}\left(f(u_n)u_n-2F(u_n)\right) \dif x&\nonumber\\
	&\leq  C_3\|u_n\|^{4\kappa}_{2,p}+3C_{\varepsilon}C_0 C_1\|u_n\|^{r}_{2,p}.
\end{align}
This also leads to a contradiction since $4\kappa >2$. The proof is completed.
\end{proof}

We now define the following two infimum values:
\begin{align*}
&\varUpsilon_m:=\inf\limits_{u\in\cB_m\cap\mathcal{P}}I(u), \quad\text{ and } \quad  \varUpsilon_{m,r}:=\inf\limits_{u\in E_r\cap\cB_m\cap\mathcal{P}}I(u).
\end{align*}
Lemma \ref{lem2.7} implies that $0\leq \varUpsilon_{m}\leq \varUpsilon_{m,r}<+\infty$.
We in fact have the equality: $\varUpsilon_m=\varUpsilon_{m,r}$.
\begin{lemma}\label{lem3.5}
Suppose that $(f_1)$-$(f_3)$ are fulfilled. Then $\varUpsilon_m=\varUpsilon_{m,r}$. If the assumption $(f_4)$ additionally holds, then
$\varUpsilon_m\to 0^+$ as $\xi\to +\infty$.
\end{lemma}
\begin{proof} We recall the definitions of $H(t)$ and  $K(t)$. Note that $F$ is even. Let $\left\{u_n\right\}\subset \cB_m \cap \mathcal{P}$ be  a minimizing sequence of $\varUpsilon_{m}$, then $\left\{|u_n|\right\}$ is also a minimizing sequence. Thus, without loss of generality, we may assume $u_n\geq 0$. Denote the Schwarz symmetrization of $\left\{u_n\right\}$ by $\left\{u_n^*\right\}$.  Then  $u_n^*$ satisfies that
\begin{align*}
	\|\nabla u_n^*\|^2_2\leq \|\nabla u_n\|^2_2,\quad  \|\nabla u_n^*\|^p_p\leq \|\nabla u_n\|^p_p, \quad \text{ and } \quad
	\|u_n^*\|_2^2 =\|u_n\|_2^2.
\end{align*}
From Lemma \ref{lem2.6}, both  $H(t)$ and $K(t)$ are monotonic functions in $(0,+\infty)$, it then follows that
\begin{equation*}
\int_{\mathbb R^2}H(u^*)\dif x=\int_{\mathbb R^2}H(u)\dif x, \text{ and }
\int_{\mathbb R^2}K(u^*)\dif x=\int_{\mathbb R^2}K(u)\dif x. 
\end{equation*}
Moreover, it is obvious that $u_n^* \in \cB_m \cap E_r$ and $P\left(u_n^*\right) \leq P\left(u_n\right)=0$.  By Lemma \ref{lem2.7}, there exists $s_n^*:=s_{u_n^*} \leq 0$ such that $P\left(s_n^* \star u_n^*\right)=0$.

On the other hand, for any $u\in \cB_m$ and  $s\in\mathbb{R}$, we have
\begin{align*}
	I(s\star u)- \frac{1}{2}P(s\star u)=&
	\frac{2-p}{p}e^{2(p-1)s}\|\nabla u\|^p_p+\frac{1}{2}e^{-2s}\int_{\mathbb{R}^2}K(e^su)\dif x\\
	=&\frac{2-p}{p}e^{2(p-1)s}\|\nabla u\|^p_p+\frac{1}{2}\int_{\mathbb{R}^2}\frac{K(e^su)}{e^{2s}u^2}u^2\dif x.
\end{align*}
This, together with Lemma \ref{lem2.6}, implies that  $2I(s\star u)-P(s\star u)$ is nondecreasing with respect to $s\in \mathbb{R}$. We thus conclude that
\begin{align*}
	\varUpsilon_m \leq \varUpsilon_{m,r} & \leq I\left(s_n^* \star u_n^*\right) =I\left(s_n^* \star u_n^*\right)-\frac{1}{2} P\left(s_n^* \star u_n^*\right) \\
	& \leq I\left(u_n^*\right)-\frac{1}{2} P\left(u_n^*\right) \\
	& =\frac{2-p}{p}\|\nabla u_n^*\|^p_p+\frac{1}{2}\int_{\mathbb{R}^2}K(u_n^*)\dif x \\
	&\leq \frac{2-p}{p}\|\nabla u_n\|^p_p+\frac{1}{2}\int_{\mathbb{R}^2}K(u_n)\dif x\\
	& =I\left(u_n\right)-\frac{1}{2}P(u_n)=I(u_n)=\varUpsilon_m+o_n(1),
\end{align*}
which implies that $\varUpsilon_m=\varUpsilon_{m,r}$.

We now prove the asymptotic behaviors of $\varUpsilon_{m}$ as $\xi\to+\infty$.
For $x\in \mathbb{R}^2$, we take
\begin{equation*}
	\zeta_0(x):=\Big(\frac{m}{\pi}\Big)^{\frac{1}{2}}e^{-\frac{\abs{x}^2}{2}},
\end{equation*}
 so that $\zeta_0\in E$ and  $\|\zeta_0\|_2^2=m$. We thus conclude that $\zeta_0\in \cB_m$. By using $(f_4)$, for any $s\in \mathbb{R}$, we deduce that
\begin{align*}
	I\left(s\star \zeta_0\right)\leq \frac{1}{2}e^{2s}\|\nabla\zeta_0\|^2_2+\frac{1}{p}e^{2(p-1)s}\|\nabla\zeta_0\|^p_p-\frac{\xi}{\nu}e^{(\nu-2)s}\int_{\mathbb{R}^2}|\zeta_0|^{\nu}\dif x.
\end{align*}
For  $ t\geq 0$ we set
\begin{equation*}
	g(t):=\frac{1}{2}A_1t^2+\frac{1}{p}A_2t^{2(p-1)}-\frac{\xi}{\nu}A_3t^{\nu-2},
\end{equation*}
with
\begin{align*}
	A_1:=\|\nabla\zeta_0\|^2_2,\quad  A_2:=\|\nabla\zeta_0\|^p_p,\quad \text{and } \quad  A_3:=\|\zeta_0\|_\nu^\nu.
\end{align*}
Then  we can easily verify that
\begin{equation*}
	g'(t)=A_1t+2(1-1/p)A_2t^{2p-3}-\xi(1-2/\nu)A_3t^{\nu-3},
\end{equation*}
and
\begin{align*}
	g''(t)=A_1+\frac{2(p-1)(2p-3)}{p}A_2t^{2(p-2)}-\frac{\xi(\nu-2)(\nu-3)}{\nu}A_3t^{\nu-4}.
\end{align*}
An elementary calculation shows that there exists a unique $t_\xi>0$ such that $g'(t_\xi)=0$, and $g''(t_\xi)<0$.
Therefore, $g(t)$ achieves the unique maximum at $t_\xi$ and
\begin{align*}
	g(t_\xi)=\frac{2-p}{p}A_2t_\xi^{2p-2}+\frac{\xi(\nu-4)}{2\nu}A_3t_\xi^{\nu-2}>0.
\end{align*}
Moreover, from $g'(t_\xi)=0$, we have $t_\xi=O(\xi^{-\frac{1}{\nu-2p}})  $ as $ \xi \to +\infty$. It then follows that
\begin{align*}
	g(t_0)=C_1\xi^{-\frac{2p-2}{\nu-2p}}+C_2\xi^{1-\frac{\nu-2}{\nu-2p}}
	=O(\xi^{-\frac{2p-2}{\nu-2p}}),\quad \text{as\ } \xi \to +\infty.
\end{align*}
Finally, in view of  Lemma \ref{lem2.7} and since $\nu>4>2p>2$, we derive that, as $\xi \to +\infty$,
\begin{equation*}
\varUpsilon_m\leq I\left(s_{\zeta_0}\star \zeta_0\right)\leq g(t_0)=O(\xi^{-\frac{2p-2}{\nu-2p}})\to 0.
\end{equation*}
\end{proof}
\begin{remark}\label{rem34}
Recalling the minimization problem \eqref{minipoho}, we  analogously define  the minimization problem
\begin{equation*}\label{minphoradial}
	\gamma_{m,r}=\inf_{u\in \cP_m\cap E_r}I(u).
\end{equation*}
Repeating the same arguments as above, we obtain $\gamma_m=\gamma_{m,r}$.
Consequently, we readily verify that 
\begin{equation*}
	0<\varUpsilon_{m}=\varUpsilon_{m,r}\leq \gamma_{m,r}=\gamma_{m}<+\infty.
\end{equation*}
On the other hand,
since $\zeta_0\in \mathcal{S}_m$, we may apply the preceding argument once again to conclude that  $\gamma_m\to 0^+$ as $\xi\to+\infty$.  As a result, the following two Lemmas \ref{lem3.3}-\ref{lem3.4} also hold for any minimizing sequence of $\gamma_m$ for sufficiently large $\xi$.
\end{remark}

\begin{lemma}\label{lem3.3}
If $\{u_n\}\subset \cB_m\cap\mathcal{P}$ is a minimizing sequence of $\varUpsilon_{m}$, then there exist $\beta_0>0$ and $\bar{\xi}_0>0$ independent of $n$ such that $\|u_n\|^2_{2,p}\leq\beta_0$ for all $\xi\geq \bar{\xi}_0$.
\end{lemma}
\begin{proof}
By contradiction, we assume that $\|u_n\|^2_{2,p}\to +\infty$ as $n\to\infty$ for some $\xi>0$. Then, up to a subsequence, two scenarios arise: either 
 \begin{equation*}
 	\lim_{n\to \infty} \|\nabla u_n\|_2 \|\nabla u_n\|_p^{-1}=C_0\in[0,+\infty), \text{  or  } \lim_{n\to \infty} \|\nabla u_n\|_2 \|\nabla u_n\|_p^{-1}=+\infty.
 \end{equation*}

For the former case, let $0<\delta_0\leq (2C_0^2+1)^{-1/2}$ and select
$s_n\in \R$ for  $n\geq 1$ satisfying
\begin{align*}
	\frac{p}{2p-2}\ln\left(\|u_n\|_{2,p}\right)\leq s_n \leq 	\frac{p}{2p-2}\ln\left(\delta_0^{-1}\|\nabla u_n\|_p\right).
\end{align*}
It follows that $s_n\to+\infty$. By setting $\omega_n:=(-s_n)\star u_n$, we deduce that
\begin{align*}
	\delta_0^2&\leq  e^{(\frac{4}{p}-4)s_n}\|\nabla u_n\|^2_p \leq \|\omega_n\|^2_{2,p}&\\
	&= e^{-2s_n}\|\nabla u_n\|^2_2+e^{(\frac{4}{p}-4)s_n}\|\nabla u_n\|^2_p
	\leq	e^{(\frac{4}{p}-4)s_n}\|u_n\|^2_{2,p}\leq 1.
\end{align*}

Set \begin{equation*}
	\sigma:=\limsup_{n\to \infty}\Big(\sup \limits_{y\in\mathbb{R}^2}\int_{B(y,1)}|\omega_n|^2\dif x\Big).
\end{equation*}
If $\sigma>0$, then, up to a subsequence, there exists $\{y_n\}\subset\mathbb{R}^2$ and $v_0\in E$ with $v_0\not\equiv0$ such that
$v_n:=\omega_n(\cdot+y_n)\rightharpoonup v_0$ weakly in  $E$, $v_n\to v_0$ strongly in $L^\nu(\mathbb R^2)$, and  $v_n\to v_0$ almost everywhere in $\mathbb R^2$. By  combining $(f_4)$ and the Fatou lemma, we deduce that
\begin{align*}
	&	\varUpsilon_m+o_n(1)=I\left(s_n\star v_n\right)\\
	&=\frac{1}{2}e^{2s_n}\|\nabla v_n\|^2_2+\frac{1}{p}e^{2(p-1)s_n}\|\nabla v_n\|^p_p-e^{-2s_n}\int_{\mathbb{R}^2}F(e^{s_n}v_n)\dif x\\
	&\leq e^{2s_n}\Big( \frac{1}{2}-\frac{1}{2}\|\nabla v_n\|^2_p+\frac{1}{p}e^{2(p-2)s_n}\|\nabla v_n\|^p_p-\frac{\xi}{\nu}e^{(\nu-4)s_n}\|v_0\|_{\nu}^{\nu} \Big)\to -\infty .
\end{align*}
This is in contradiction to Lemma \ref{lem3.5}.

On the other hand, if $\sigma=0$, we  define $z_n:=(2\alpha_0/{\pi})^{1/2}(s\star \omega_n)$.
When
\begin{equation*} s\leq s_0:= \min\left\{0,\frac{p}{4p-4}\ln\left(\pi/{(2\alpha_0)}\right)\right\},
\end{equation*}
we deduce
\begin{align*}
	\|z_n\|^2_{2,p}&=\frac{2\alpha_0}{\pi}\left(e^{2s}\|\nabla \omega_n\|^2_2+e^{(4-\frac{4}{p})s}\|\nabla \omega_n\|^2_p\right)\\
	&\leq \frac{2\alpha_0}{\pi}\left(e^{2s}\|\nabla \omega_n\|^2_2+e^{(4-\frac{4}{p})s}\left(1-\|\nabla \omega_n\|^2_2\right)\right)\\
	&=\frac{2\alpha_0}{\pi}e^{(4-\frac{4}{p})s}+\frac{2\alpha_0}{\pi}\|\nabla \omega_n\|^2_2\left(e^{2s}-e^{(4-\frac{4}{p})s}\right)\leq 1.
\end{align*}
This, together with Lemma $\ref{lem2.1}$ implies that  there exists $C_1>0$ such that
\begin{align*}
	\int_{\mathbb{R}^2}\varPhi_{4\alpha_0,j_0}\left(s\star \omega_n\right)\dif x
	&=\int_{\mathbb{R}^2}\varPhi_{2\pi,j_0}(z_n)\dif x\leq C^2_1.
\end{align*}
Using similar arguments as in Lemma $\ref{lem2.4}$, for $\alpha=2\alpha_0$ and $r>p^*$, and for any $\varepsilon>0$,  there exists  $C_{\varepsilon}>0$ such that
\begin{align*}
	\int_{\mathbb{R}^2}F(u)\dif x
	\leq\varepsilon \|u\|^{4}_{4}+C_{\varepsilon}\|u\|_{2r}^r \left(\int_{\mathbb{R}^2}\varPhi_{4\alpha_0,j_0}(u)\dif x\right)^{1/2}.
\end{align*}
Since $\sigma=0$, 	from Lions' lemma \cite{W96}, we have $\omega_n\to 0$ strongly in $L^k(\mathbb{R}^2)$ for any $k > 2$.  Therefore, for fixed $s\leq s_0$, we deduce that, as $n\to\infty$,
\begin{align*}
	\int_{\mathbb{R}^2}|F(s\star \omega_n)|\dif x
	&\leq \varepsilon e^{2s}  \|\omega_n\|^{4}_{4}+C_{\varepsilon}C_1 e^{(r-1)s}\|\omega_n\|_{2r}^{r}\to 0 .
\end{align*}
Consequently, we can utilize  Lemma $\ref{lem2.7}$ to conclude that
\begin{align*}
	\varUpsilon_m+o_n(1)&=I(u_n) \geq I((s_0-s_n)\star u_n)\\
	&=\frac{1}{2}e^{2s_0}\|\nabla \omega_n\|^2_2+\frac{1}{p}e^{2(p-1)s_0}\|\nabla \omega_n\|^p_p-\int_{\mathbb{R}^2}F(s_0\star \omega_n)\dif x\\
	&\geq \frac{1}{2}e^{2s_0}\|\nabla \omega_n\|^2_2+\frac{1}{p}e^{2(p-1)s_0}\|\nabla \omega_n\|^2_p+o_n(1)\\
	&\geq\frac{1}{2}e^{2s_0}\|\omega_n\|^2_{2,p}+o_n(1)\\
	&\geq \min\Big\{\frac{1}{2}\delta^2_0,\frac{1}{2}\delta^2_0\Big(\frac{\pi}{2\alpha_0}\Big)^{\frac{p}{2p-2}}\Big\}
	+o_n(1),
\end{align*}
which also contradicts with Lemma \ref{lem3.5} when $\xi$ becomes sufficiently large.

For the latter case,  we still define  $\omega_n=(-s_n)\star u_n$ where $s_n\in \R$ satisfies
\begin{equation*}s_n=\ln\left(\|\nabla u_n\|_2\right).\end{equation*}
It is evident that $\|\nabla \omega_n\|_2=e^{-s_n}\|\nabla u_n\|_2=1$ and  $\lim_{n\to\infty}s_n= +\infty$. In addition, for sufficiently large $n$,  we have
\begin{align*}
	\|\nabla \omega_n\|^p_p=e^{(2-2p)s_n}\|\nabla u_n\|^p_p\leq e^{(2-2p)s_n}\|\nabla u_n\|^p_2.
\end{align*}
Similarly, as in the previous case, we examine again the concentration limit behavior of $\{\omega_n\}$. Assuming $\sigma>0$, then, up to a subsequence and a translation, there exists $u_0\in E\setminus\{0\}$ such that $u_n\to u_0$ strongly in $L^\nu(\mathbb R^2)$, and  $u_n\to u_0$ almost everywhere in $\mathbb R^2$. Applying an analogous argument with reliance on  $(f_4)$ alongside the Fatou lemma, we deduce that
\begin{align*}
	\varUpsilon_m+o_n(1)=I(u_n)&=I\left(s_n\star \omega_n\right)\\
	&=\frac{1}{2}e^{2s_n}\|\nabla\omega_n\|^2_2+\frac{1}{p}e^{2(p-1)s_n}\|\nabla\omega_n\|^p_p-e^{-2s_n}\int_{\mathbb{R}^2}F(e^{s_n}\omega_n)\dif x\\
	&\leq e^{2s_n}\Big(\frac{1}{2}+\frac{1}{p}\|\nabla u_n\|^{p}_2-\frac{\xi}{\nu}e^{(\nu-4)s_n}\|v_0\|_{\nu}^{\nu}\Big) \to -\infty .
\end{align*}
This, combined with Lemma \ref{lem3.5}, would lead to an absurdity for
sufficiently large $\xi$.
Finally,  if $\sigma=0$, we have $\omega_n\to 0$ strongly in $L^k(\R^2)$ for any $k>2$.
Again, we set $z_n:=({2\alpha_0}/{\pi})^{1/2}(s\star \omega_n)$. We are ready to verify that for all $ s\leq s_0:= \ln\left({\pi}/{(2\alpha_0)}\right)^{1/2}$,
\begin{align*}
	\|\nabla z_n\|^2_{2} =\frac{2\alpha_0}{\pi}e^{2s}\|\nabla \omega_n\|^2_2
	\leq 1, \text{ and }\
	\|z_n\|_2^2=\frac{2\alpha_0}{\pi}\|\omega_n\|^2_2\leq \frac{2\alpha_0m}{\pi}.
\end{align*}
From Lemma $\ref{lem2.1}$, we find that for all $s\leq s_0$, there exists $C_2>0$ such that
\begin{align*}
	\int_{\mathbb{R}^2}\big(e^{4\alpha_0|s\star\omega_n|^2}-1\big)\dif x=\int_{\mathbb{R}^2}\big(e^{2\pi|z_n|^2}-1\big)\dif x\leq C^2_2.
\end{align*}
We thus have by $(f_1)$-$(f_2)$ that, as $n\to\infty$,
\begin{align*}
	\int_{\mathbb{R}^2}|F(s\star \omega_n)|\dif x\leq \varepsilon e^{2s}  \|\omega_n\|^{4}_{4}+C_{\varepsilon}C_2e^{(r-1)s}\|\omega_n\|_{2r}^r\to 0 .
\end{align*}
By invoking Lemma $\ref{lem2.7}$, we obtain that
\begin{align*}
	\varUpsilon_m+o_n(1)=I(u_n)\geq& I((s_0-s_n)\star u_n)=I(s_0\star \omega_n)\\
	=&\frac{1}{2}e^{2s_0}\|\nabla \omega_n\|^2_2+\frac{1}{p}e^{2(p-1)s_0}\|\nabla \omega_n\|^p_p-\int_{\mathbb{R}^2}F(s_0\star \omega_n)\dif x\\
	\geq& \frac{1}{2}e^{2s_0}\|\nabla \omega_n\|^2_2+o_n(1)=\frac{\pi}{4\alpha_0}+o_n(1),
\end{align*}
which contradicts to Lemma \ref{lem3.5} again for largely $\xi$.
\end{proof}

\begin{lemma}\label{lem3.4}
If $\{u_n\}\subset \cB_m\cap\mathcal{P}$ is a minimizing sequence of $\varUpsilon_{m}$, then there exists  $\bar{\xi}\geq \bar{\xi}_0>0$  such that  for all $n\in \mathbb N$,
\begin{equation*}\|\nabla u_n\|_{2}^2\leq \frac{\pi}{2\alpha_0}
	\end{equation*}
whenever $\xi\geq \bar{\xi}$.
\end{lemma}	
The idea for the proof is inspired by \cite{CLY23}, yet we provide rigorous details for completeness and to enhance readability.
\begin{proof}
Let $t_0:=\ln\left({\pi}/{(2\alpha_0)}\right)^{1/2}$ and $t_k:=\ln\left(1+\beta_0+k\right)$, where $k\geq 1$ will be determined later and $\beta_0$ is defined as in Lemma $\ref{lem3.3}$. Considering $(f_1)$-$(f_2)$, for any $\varepsilon>0$, there exists $C_\varepsilon>0$ such that
\begin{equation}\label{eq3.2}
	|F(t)| \leq \varepsilon|t|^4+\varepsilon|t|^{\nu+1}\big(e^{2\alpha_0 t^2}-1\big)+C_\epsilon|t|^{\nu},\  \forall t \in \mathbb{R}.
\end{equation}
Setting
$v_n:= ({2 \alpha_0}/{\pi})^{1/2}((t_0-t_k)\star u_n)$, we are ready to verify that
\begin{align*}
	\|v_n\|_2^2 \leq \frac{2 \alpha_0 m}{\pi}, \; \text{ and }\; \|\nabla v_n \|_2^2=\frac{2 \alpha_0}{\pi}e^{2\left(t_0-t_k\right)}  \|\nabla u_n\|^2_2 \leq \frac{\|\nabla u_n\|^2_2}{(\beta_0+1)^2} \leq 1  .
\end{align*}
By employing Lemma $\ref{lem2.1}$, we deduce there exists $C_0>0$ such that
\begin{align*}
	\int_{\mathbb{R}^2}\big(e^{4 \alpha_0\left|\left(t_0-t_k\right) \star u_n\right|^2}-1\big) \dif x=\int_{\mathbb{R}^2}\big(e^{2 \pi\left|v_n\right|^2}-1\big) \dif x \leq C^2_0.
\end{align*}
Combing the continuous embedding $H^{1}(\R^2)\hookrightarrow L^k(\R^2)$ for $k\geq 2$ and  the
Gagliardo--Nirenberg interpolation inequality in $H^1(\R^2)$ that for any $r\in (2,\infty)$, there exists $C_{GN}=C(r)$ such that
\begin{equation*}\|u\|_r^r\leq C_{GN}\|\nabla u\|_2^{r-2}\|u\|_2^2.\end{equation*}
We then obtain $C_1, C_2,C_3>0$ that depend on $m$ but are independent of $n$, such that for $\{u_n\}\subset \cB_m$,
\begin{align*}
	&\|u_n\|^4_4 \leq C_1  \|\nabla u_n\|^2_2,\|u_n\|^{2\nu+2}_{2\nu+2}  \leq C^2_2 \|\nabla u_n\|^{2\nu}_2,\|u_n\|^{\nu}_{\nu}  \leq C_3 \|\nabla u_n\|^{\nu-2}_2.
\end{align*}
Therefore, in view of $\eqref{eq3.2}$, Lemma $\ref{lem3.3}$ and  the fact that $\nu>4$, for any $\varepsilon\in (0,1)$, we conclude that
\begin{align*}\label{3.3}
	& \int_{\mathbb{R}^2}\left|F\left(\left(t_0-t_k\right) \star u_n\right)\right| \dif x &\nonumber\\
	\leq& \varepsilon e^{2\left(t_0-t_k\right)} C_1  \left\|\nabla u_n\right\|_2^2+\varepsilon e^{\nu\left(t_0-t_k\right)}C_0 C_2  \left\|\nabla u_n\right\|_2^{\nu} +e^{(\nu-2)\left(t_0-t_k\right)}C_{\varepsilon}C_3\left\|\nabla u_n\right\|_2^{\nu-2} &\nonumber\\
	\leq& C_4 e^{2\left(t_0-t_k\right)}\left\|\nabla u_n\right\|_2^2 ,
\end{align*}
where
\begin{align*}
	C_4:&= \varepsilon C_1+\Big(\frac{\pi}{2 \alpha_0}\Big)^{\frac{\nu-2}{2}}\frac{ \varepsilon C_0 C_2 \beta_0^{\frac{\nu-2}{2}}}{(\beta_0+1+k)^{\nu-2}}+\Big(\frac{\pi}{2 \alpha_0}\Big)^{\frac{\nu-4}{2}}\frac{C_{\varepsilon}C_3\beta_0^{\frac{\nu-4}{2}}}{(\beta_0+1+k)^{\nu-4}}\\
	&\leq   \varepsilon C_1+\varepsilon C_0 C_2 \Big(\frac{\pi}{2k \alpha_0}\Big)^{\frac{\nu-2}{2}}  + C_{\varepsilon}C_3 \Big(\frac{\pi}{2k \alpha_0}\Big)^{\frac{\nu-4}{2}}
\end{align*}
Note that $\nu>4$.  Fixing small $\varepsilon$ and taking sufficiently large $k$ such that $C_4\leq 1/4$ so that we have
\begin{align*}
	\int_{\mathbb{R}^2}\left|F\left(\left(t_0-t_k\right) \star u_n\right)\right| \dif x \leq \frac{1}{4} e^{2\left(t_0-t_k\right)}\left\|\nabla u_n\right\|_2^2.
\end{align*}
In light of Lemma \ref{lem2.7}, we then derive by fixing the choice of the above $k$ that
\begin{align*}
	\varUpsilon_m&+o_n(1)  =I\left(u_n\right) \geq I\left(\left(t_0-t_k\right) \star u_n\right) \\
	& =\frac{1}{2} e^{2\left(t_0-t_k\right)} \left\|\nabla u_n\right\|^2_2 +\frac{1}{p}e^{2(p-1)(t_0-t_k)}\left\|\nabla u_n\right\|^p_p-\int_{\mathbb{R}^2} F\left(\left(t_0-t_k\right) \star u_n\right) \dif x \\
	& \geq
	\frac{1}{4} e^{2\left(t_0-t_k\right)}\left\|\nabla u_n\right\|_2^2.
\end{align*}
Regarding Lemma $\ref{lem3.5}$, there exists sufficiently large $n_0$, such that
\begin{align*}
	\left\|\nabla u_n\right\|_2^2 \leq 8 \varUpsilon_m e^{-2\left(t_0-t_k\right)} \leq  \frac{16 \alpha_0}{\pi}(\beta_0+1+k)^2\varUpsilon_{m}, \quad \forall n \geq n_0.
\end{align*}
Note that $\varUpsilon_{m}\to 0^+$ as $\xi\to+\infty$. Therefore,  by omitting the first finitely many terms if necessary, we can find a sufficiently large $\bar{\xi}\geq \bar{\xi}_0>0$  such that 
\begin{equation*}
	\|\nabla u_n\|^2_2\leq \frac{\pi}{2\alpha_0},
\end{equation*} 
for all $n\in \mathbb N$ whenever $\xi\geq\bar{\xi}_0$.
\end{proof}

\begin{remark}
It follows from the previous argument that
$\|\nabla u_n\|_2$  becomes arbitrarily small as $\xi$ grows sufficiently large. A similar argument applies to $\|\nabla u_n\|_p$, and hence  $\|u_n\|_{2,p}\to 0$ as $\xi\to+\infty$.  Combining Lemma~ \ref{lem2.7}, we conclude that the lower bound $\tau_0$ for $u\in \cB_m\cap \mathcal{P}$ actually depends on $\xi$.
\end{remark}

Observing from the fact that $\cS_m\subset \cB_m$, so that $\cP_m\subset \cB_m\cap \cP$. We first obtain $\gamma_m\geq \varUpsilon_m$. In fact, we can prove the reversed inequality for largely $\xi$.
\begin{lemma} \label{equalvalues}
Suppose that $f$ satisfies $(f_1)$-$(f_4)$. Then there exists $\xi_0\geq \bar{\xi}$ such that $\varUpsilon_{m}$ is achieved at some $u\in \cP_m$. Consequently, $\gamma_m=\varUpsilon_m$ for  $\xi\geq \xi_0$.
\end{lemma}
\begin{proof} Let $\left\{u_n\right\}\subset  \cB_m \cap \mathcal{P}\cap E_r$ be a minimizing sequence for $\varUpsilon_{m,r}$. By Lemma $\ref{lem3.5}$, we  know that $\left\{u_n\right\} \subset E$ is also a minimizing sequence for $\varUpsilon_{m}$.  Applying Lemma \ref{lem3.4} and omitting the first finitely many terms if necessary, we deduce that for all $n\in \mathbb N$, whenever $\xi \geq \bar{\xi} $,
\begin{equation*}
	\|u_n\|^2_{2,p}\leq \beta_0,\ \text{ and }\
	\|\nabla u_n\|^2_2\leq \frac{\pi}{2 \alpha_0}.
\end{equation*}
This, together with  the fact that $u_n\in \cB_m$,  implies that $\left\{u_n\right\}\subset E_r$ is bounded. As a result, there exists $u \in E_r$ such that $u_n \rightharpoonup u$ in $E_r$. Furthermore, Lemma \ref{lem2.2} implies that $u_n \to u$ in $L^k\left(\mathbb{R}^2\right)$ for all $k>2$ and $u_n(x)$ converges to  $u(x)$ almost everywhere in  $\mathbb{R}^2$.

The subsequent proof will be carried out into several steps.

{\bf Step 1}. The weak limit $u\neq 0$. We first establish the convergence equalities needed for the subsequent analysis:
\begin{align}\label{capfconv}
	\lim_{n\to\infty}\int_{\mathbb{R}^2} F\left(u_n\right) \dif x =\int_{\mathbb{R}^2} F\left(u\right) \dif x,
	\text{
		and }
	\lim_{n\to\infty}	\int_{\mathbb{R}^2} f\left(u_n\right) u_n \dif x =\int_{\mathbb{R}^2} f\left(u\right) u \dif x.
\end{align}
Having obtained these,  we are able to conclude that $u:=u_\xi\neq 0$. In fact, by combining \eqref{capfconv} and the fact that $P(u_n)=0$, we have, as $n\to\infty$,
\begin{equation}\label{eq3.5}
	\|\nabla u_n\|_2^2+2(1-\frac{1}{p})\|\nabla u_n\|_p^p=\int_{\R^2}H(u_n)\dif x\to \int_{\R^2}H(u) \dif x.
\end{equation}
If $u=0$, then from \eqref{eq3.5} we see that $\|u_n\|_{2,p}\to 0$, which contradicts with Lemma \ref{lem2.7}.

Let $\Theta$ denote either $F(t)$ or  $f(t)t$. The desired convergence \eqref{capfconv} follows once we prove $\Theta(u_n)\to \Theta(u)$ in $L^1(\mathbb R^2)$. To this end, we choose $r>1$ and define for $t\in \mathbb R$ that
\begin{equation*}
	Q(t):=|t|^{4}+\abs{t}^r (e^{2\alpha_0|t|^2}-1).
\end{equation*}
Since $\|\nabla u_n\|_2^2\leq \pi/(2\alpha_0)$, then using the boundedness of $\{u_n\}\subset E$ and the embedding result in Lemma \ref{lem2.2}, we obtain from  Lemma $\ref{lem2.1}$,  that
\begin{align*}
	\int_{\mathbb R^2}Q(u_n)\dif x &\leq  \|u_n\|_4^4+\|u_n\|_{2r}^{r}\Big(\int_{\mathbb R^2}(e^{ 4\alpha_0\abs{u_n}^2}-1)\dif x\Big)^{1/2}\\
	&\leq   C_1\|u_n\|_E^4+C_2\|u_n\|_E^r\Big(\int_{\mathbb R^2}(e^{2\pi \frac{4\alpha_0}{2\pi}\abs{u_n}^2}-1)\dif x\Big)^{1/2}\leq C_3.
\end{align*}

Note that $\Theta$ is continuous and satisfies $(f_1)$ and $(f_2)$. Then for $\varepsilon>0$ and $q_0>2$, there exists $C_{\varepsilon}>0$ such that
\begin{align*}
	|\Theta(t)| \leq \varepsilon Q(t)+C_{\varepsilon}|t|^{q_0}, \quad \text { for all } t \in \mathbb{R}.
\end{align*}
Consequently, we have  for all $n\in \mathbb N$,
\begin{align*}
	\int_{\mathbb{R}^2}\left|\Theta(u_n)\right| \dif x \leq \varepsilon \int_{\mathbb{R}^2} Q\left(u_n\right) \dif x+C_{\varepsilon} \int_{\mathbb{R}^2}\left|u_n\right|^{q_0}\dif x \leq \varepsilon C_3+C_{\varepsilon} C_4.
\end{align*}
We therefrom conclude by Fatou's lemma that $\Theta(u)\in L^1(\mathbb R^2)$.

We now complete the proof by adopting the argument in \cite{BrezisLieb83}.  We set  for  almost every $x \in \mathbb{R}^2$ and $n\in \mathbb N$ that
\begin{align*}
	S_{\varepsilon, n}(x):=\left(\left|\Theta\left(u_n(x)\right)-\Theta(u(x))\right|-\varepsilon Q\left(u_n(x)\right)\right)^{+},
\end{align*}
with the notation $w^+(x):=\max\{w(x),0\}$.
Clearly, we observe that  for all  $n \in \mathbb{N}$,
\begin{align*}
	0 \leq S_{\varepsilon, n}(x) \leq C_{\varepsilon}\left|u_n(x)\right|^{q_0}+|\Theta(u(x))| \quad \text  { for almost every }    x \in \mathbb{R}^2.
\end{align*}
Since $u_n \rightarrow u$ in $L^{q_0}\left(\mathbb{R}^2\right)$, there exists $h_0 \in L^{q_0}\left(\mathbb{R}^2\right)$ such that, up to a subsequence, $\abs{u(x)}, \abs{u_n(x)}\leq h_0(x)$ for almost every $x \in \mathbb{R}^2$ and for all $n \in \mathbb{N}$. Therefore,  we have
\begin{equation*}
	S_{\varepsilon, n}(x) \leq C_{\varepsilon} \abs{h_0(x)}^{q_0}+|\Theta(u(x))|\in L^1(\mathbb R^2).
\end{equation*}
Moreover, by exploiting the continuity of $\Theta$ and $Q$, we know that $S_{\varepsilon, n}(x) \rightarrow 0$ almost everywhere in $\mathbb{R}^2$ as $n \rightarrow \infty$. The Lebesgue's dominated convergence theorem gives us that
\begin{align*}
	\lim _{n \rightarrow \infty} \int_{\mathbb{R}^2} S_{\varepsilon, n}(x) \dif x=0
\end{align*}
Consequently, we have
\begin{align*}
	&\limsup _{n \rightarrow\infty} \int_{\mathbb{R}^2}\left|\Theta\left(u_n\right)-\Theta(u)\right| \dif x&\\
	&\leq \limsup _{n \rightarrow\infty} \int_{\mathbb{R}^2} S_{\varepsilon, n}(x) \dif x+\varepsilon \limsup _{n \rightarrow \infty} \int_{\mathbb{R}^2} Q\left(u_n\right) \dif x \leq \varepsilon C_2.
\end{align*}
Since $\varepsilon>0$ is arbitrary, we obtain the assertion \eqref{capfconv}.

{\bf Step 2}. The weak limit $u\in \cB_m\cap \cP$ and achieves the value
$\varUpsilon_m$, i.e.,
\begin{equation*}
	\varUpsilon_m=I\left(u\right)>0.
\end{equation*}

\indent Since $\{u_n\}\subset \cB_m\cap \mathcal{P}$, by considering \eqref{eq3.5} and employing the weakly lower semi-continuity of the norm, we obtain $u\in \cB_m$ and $P(u)\leq 0$.
Indeed, we have
\begin{align*}
	\|\nabla u\|^2_2+\frac{2(p-1)}{p}\|\nabla u\|^p_p&\leq\liminf\limits_{n\to \infty}\Big(	\|\nabla u_n\|^2_2+\frac{2(p-1)}{p}\|\nabla u_n\|^p_p\Big)\\
	&=\mathop{\lim}\limits_{n\to \infty}\int_{\mathbb{R}^2}H(u_n)\dif x=\int_{\mathbb{R}^2}H(u)\dif x.
\end{align*}
If $P(u)<0$, we claim that there exists $t_0\in (0,1)$ such that $P(u(\cdot/t_0))=0$. In fact, let
\begin{align*}
	g(t):=P\left(u\left({x}/{t}\right)\right)=\|\nabla u\|^2_2+\frac{2(p-1)}{p}t^{2-p}\|\nabla u\|^p_p - t^2\int_{\mathbb{R}^2}H(u)\dif x.
\end{align*}
Note that $p<2$. We clearly have $g(t)\to \|\nabla u\|_2^2$ as $t\to 0$ and  $g(t)\to -\infty$ as $t\to +\infty$.
Then we verify that   
\begin{equation*}
	g^{\prime}(t)=At^{1-p}-Bt,
\end{equation*}
with 
\begin{equation*}
	A:=\frac{2(p-1)(2-p)}{p}\|\nabla u\|^p_p>0
\quad \text{ and }\quad
B:=2\int_{\mathbb{R}^2}H(u)\dif x>0.
\end{equation*}
Then there exists a unique $t_*>0$ such that $g^{\prime}(t_*)=0$, and $t_*=\left({A}/{B}\right)^{{1}/{p}}$. In particular, from $P(u)<0$, we see that $A<B$, we thus obtain $t_*<1$. Therefore, we conclude that $g(t)$ is increasing on $(0,t_*)$,  decreasing on $(t_*,+\infty)$, and $g(t_*)>0$. This, together with  $g(1)<0$, implies that there exists  $t_0\in[t_*,1)$ such that $g(t_0)=0$. The claim is proved.

Set $v_0:=u(\cdot/t_0)$. By noting that $\|v_0\|_2^2=t_0^2\|u\|_2^2\leq m$, we thus have  $v_0\in \cB_m\cap \mathcal{P}$.
In light of  \eqref{capfconv} and the weakly lower semi-continuity of the norm, we deduce that
\begin{align*}
	\varUpsilon_m\leq I\left(v_0\right)&=I\left(v_0\right)-\frac{1}{2}P\left(v_0\right)\\
	&=\frac{2-p}{p}t_0^{2-p}\|\nabla u\|^p_p +\frac{1}{2}t_0^2\int_{\mathbb{R}^2}\left(f(u)u-4F(u)\right)\dif x\\
	&< \frac{2-p}{p}\|\nabla u\|^p_p +\frac{1}{2}\int_{\mathbb{R}^2}\left(f(u)u-4F(u)\right)\dif x\\
	&\leq \mathop{\lim\inf}\limits_{n \to \infty}\left(\frac{2-p}{p}\|\nabla u_n\|^p_p +\frac{1}{2}\int_{\mathbb{R}^2}\left(f(u_n)u_n-4F(u_n)\right)\dif x\right)\\
	&=\mathop{\lim}\limits_{n \to \infty}\Big(I(u_n)-\frac{1}{2}P(u_n)\Big)=\mathop{\lim}\limits_{n \to \infty}I(u_n)=\varUpsilon_m,
\end{align*}
which leads to a contradiction. Therefore, we have $P(u)=0$, and then we obtain
\begin{align*}
	\varUpsilon_m\leq I\left(u\right)&=I\left(u\right)-\frac{1}{2}P\left(u\right)\\
	&= \frac{2-p}{p}\|\nabla u\|^p_p +\frac{1}{2}\int_{\mathbb{R}^2}\left(f(u)u-4F(u)\right)\dif x\\
	&\leq \mathop{\lim\inf}\limits_{n \to \infty}\left(\frac{2-p}{p}\|\nabla u_n\|^p_p +\frac{1}{2}\int_{\mathbb{R}^2}\left(f(u_n)u_n-4F(u_n)\right)\dif x\right)\\
	&=\mathop{\lim}\limits_{n \to \infty}\Big(I(u_n)-\frac{1}{2}P(u_n)\Big)=\mathop{\lim}\limits_{n \to \infty}I(u_n)=\varUpsilon_m.
\end{align*}
Consequently, all of the above inequalities become to be equalities. We  thus conclude that $u\in \cB_m\cap \mathcal{P}$ satisfies $I(u)=\varUpsilon_m>0$ and $u_n$ converges strongly to $u$ in $E^{2,p}$.

{\bf Step 3}. We prove that $u=u_\xi \in \mathcal{P}_m$  for sufficiently large $\xi$, so that $I(u)\geq \gamma_m$. The conclusion that $\gamma_m=\varUpsilon_m$ then immediately follows.

By contradiction, suppose there exists a sequence $\xi_n\to +\infty$ such that $\|u_{\xi_n}\|_2^2<m$. For notational simplicity, we still denote  $u_{\xi}$ by $u$ and assume that $\|u\|_2^2<m$ for all $\xi\geq \xi_0$. Since $\left(\cB_m \setminus\mathcal{S}_m\right) \cap \mathcal{P}$ is an open subset of $\mathcal{P}$, there exists a Lagrange multiplier $\mu \in \mathbb{R}$ such that $u$ weakly solves
\begin{equation}\label{4.3}
	-\Delta u-\Delta_p u-f\left(u\right)+\mu\left(-2\Delta u-2(p-1)\Delta_p u- h\left(u\right)\right)=0.
\end{equation}
Note that $(1+2\mu)$ and $1+2(p-1)\mu$ can not be zero simultaneously.  Testing the  equation \eqref{4.3} against $u$, we obtain  that 	
\begin{equation}\label{eq4.4}
	(1+2\mu)\|\nabla u\|^2_2+(1+2(p-1)\mu)\|\nabla u\|_p^p=\int_{\R^2}f\left(u\right)u+\mu h(u)u \dif x.
\end{equation}
By using the fact that
\begin{equation}\label{enerequality}
	\varUpsilon_{m}=I(u)=\frac{1}{2}\|\nabla u\|_2^2+\frac{1}{p}\|\nabla u\|_p^p-\int_{\R^2}F(u)\dif x.
\end{equation}
We thus have
\begin{align}\label{pmcontra}
	&\big((1+2(p-1)\mu)-2(1+2\mu)/p\big)\|\nabla u\|_p^p&\nonumber \\
	&=-2(1+2\mu)\varUpsilon_m+\int_{\R^N} \big(H(u)+\mu (h(u)u-4 F(u)) \big)\dif x.
\end{align}
We now prove that for sufficiently large $\xi$, there has no possibility for $\mu$ to satisfy $1/\mu\in (-2, -{2(p-1)})$.
Otherwise, we have $(1+2(p-1)\mu)>0$ and $ 1+2\mu<0$, so that the left hand side of \eqref{pmcontra} is strictly positive. On the other hand,  by using the fact that $f(t)t\geq 4F(t)$, we are readily verify that
$H(u)=f(u)u-2F(u)\geq 2F(u)$.
This, together with   $h(u)u\geq 4H(u)$, implies that
\begin{equation*} h(u)u-4F(u) \geq 4 (H(u)-F(u))\geq 2H(u).\end{equation*}
We therefrom conclude that
\begin{equation*}\int_{\R^N} \big(H(u)+\mu (h(u)u-4 F(u))\big) \dif x \leq (1+2\mu)\int_{\R^N}H(u) \dif x. 
\end{equation*}
By invoking \eqref{enerequality}, \eqref{pmcontra} and the Pohozaev identity $P(u)=0$,we deduce that
\begin{align*}\label{pmcontra1}
	&(1+2(p-1)\mu)\|\nabla u\|_p^p-(1+2\mu)\left(2/p\right)\|\nabla u\|_p^p&\nonumber \\
	&\leq -2(1+2\mu)\varUpsilon_m+(1+2\mu)\int_{\R^2 }H(u)\dif x\\
	&=2(1+2\mu) \Big((1-2/p)\|\nabla u\|_p^p+\int_{\R^2}  F(u)\dif x\Big) .
\end{align*}
This leads to a contradiction, as shown in the following inequality:
\begin{align*}
	(1+2(p-1)\mu)\|\nabla u\|_p^p-2 (1+2\mu)\left(1-1/p\right)\|\nabla u\|_p^p
	\leq 2(1+2\mu) \int_{\R^2}  F(u)\dif x,
\end{align*}
since the left-hand side is strictly positive, whereas the right-hand side is non-positive. Hence, $(1+2\mu)$ and $(1+2(p-1)\mu)$ can not have opposite signs. Note that both $f(t)t$ and $f'(t)t$ satisfy the estimate of \eqref{keyesti}. By repeating the arguments as in Theorem \ref{phothm}, we can obtain that $u\in L^\infty_{{\rm {loc}}}(\R^2) $ and $u\in C_{{\rm{loc}}}^{1,\gamma}(\R^2)$ satisfies the corresponding Pohozaev identity:
\begin{equation}\label{newpho}
	(1+2(p-1)\mu)\left({2-p}\right)\|\nabla u\|_p^p=2p\int_{\R^2}\big(F(u)+\mu H(u)\big)\dif x.
\end{equation}
By combining \eqref{eq4.4},\eqref{newpho} and  the fact that $u \in \mathcal{P}$, we deduce that
\begin{align*}
	4\mu\left(p-1\right)(p-2)\|\nabla u\|^p_p=\mu p\int_{\mathbb{R}^2}\big(h(u)u-4 H(u)\big)\dif x,
\end{align*}
which combines with $(f_3)$ and the fact that $p<2$ implies $\mu=0$.
Then by using the fact $H(u)\geq 2F(u)$ and by noting the assumption $(f_4)$ again, we deduce from  \eqref{pmcontra} that
\begin{equation*}
\left(1-\frac{2}{p}\right)\|\nabla u\|_p^p+2\varUpsilon_m=\int_{\R^N} H(u) \dif x\geq 2\xi \int_{\R^N}\abs{u}^\nu \dif x.
\end{equation*}
Note that the left-hand side is bounded. This forces that $\|u_\xi\|_\nu\to 0$ as $\xi \to +\infty$.
By combining \eqref{eq3.2} with interpolation inequalities and embedding inequalities, we obtain,
\begin{equation*}
\int_{\mathbb R^2} F(u_\xi)\dif x \leq \int_{\mathbb R^2}f(u_\xi)u_\xi \dif x \to 0, \; \text{ as }  \xi\to+\infty.
\end{equation*}
This in turn implies $\int_{\mathbb R^2} H(u_\xi) \dif x\to 0$. Since $u_\xi\in \mathcal{P}$, it follows that
$\|u_\xi\|_{2,p}\to 0$ as $\xi\to +\infty$.
However, carrying out the same argument in \eqref{upperbound}--\eqref{lowerbound} that gave the positive lower bound $\tau_0$, leads to a contradiction, although $\tau_0$ itself depends on $\xi$.  We then conclude the existence of $\xi_0\geq \bar{\xi}$ such that $\|u\|_2^2=m$ for all $\xi\geq \xi_0$.
\end{proof}

\section{Proof of Theorem \ref{thm1}}
\label{sect4}

In order to find a Palais-Smale sequence for the infimum value of $\varUpsilon_{m}$, we introduce the  composite functional
$\varPsi(u): E\setminus\{0\}\to \R$ defined by
\begin{align*}
\varPsi(u):=I(s_u\star u)=\frac{1}{2}e^{s_u}\|\nabla u\|^2_2+\frac{1}{p}e^{2(p-1)s_u}\|\nabla u\|^p_p-e^{-2s_u}\int_{\R^2}F\left(e^{s_u}u\right)\dif x,
\end{align*}
where $s_u\in\R$ is the unique number guaranteed by Lemma $\ref{lem2.7}$.  Despite having only proven the continuity of $s_u$, the subsequent lemma illustrates that the composition functional $\varPsi$ exhibits a $C^1$ regularity. This technique initially derives from \cites{SW1,SW2} and was later developed  in \cite{JeanjeanLu2020} within the normalized framework.

\begin{lemma}\label{comp1}
Suppose that $f$ satisfies $(f_1)$-$(f_3)$. Then the functional $\varPsi$ is of class $C^1$. Moreover, for any $u\in E\setminus\{0\}$ and $\phi\in E$, it holds that
\begin{align*}
	&d\varPsi(u)\phi =dI(s_u\star u)(s_u\star \phi)&\\
	&=e^{2s_u}\int_{\R^2}\nabla u\cdot \nabla \phi\dif x+e^{2(p-1)s_u}\int_{\R^2}|\nabla u|^{p-2}\nabla u \nabla \phi\dif x -e^{-2s_u}\int_{\R^2}f(e^{s_u}u)e^{s_u}\phi\dif x.
\end{align*}
Here and in the sequel, the notation $d\varPsi$ refers to the Fr\'echet derivative of $\varPsi$.
\end{lemma}
\begin{proof}
Let $u\in E\setminus\{0\}$ and $\phi\in E$. Then, for sufficiently small $\abs{t}$,  we have $u+t\phi \in E\setminus\{0\}$. By Lemma \ref{lem2.7}, the functional
$I$ admits a unique global maximizer $s_{u+t\phi}$ along the scaling path of $u+t\phi$. To simplify notation, we set $\rho(t):=s_{(u+t\phi)}$ 	with the convention that $\rho(0)=s_u$.
We next  estimate the difference
\begin{align*}
	\varPsi(u+t\phi)-\varPsi(u)=I(\rho(t)\star(u+t\phi))-I(\rho(0)\star u).
\end{align*}
From the mean value theorem, we obtain that
\begin{align*}
	&I(\rho(t) \star(u+t\phi))-I(\rho(0)\star u)\leq I(\rho(t)\star(u+t\phi))-I(\rho(t)\star u)\\
	=&\frac{1}{2}e^{2\rho(t)}\left(\|\nabla (u+t\phi)\|^2_2-\|\nabla u\|^2_2\right)+\frac{1}{p}e^{2(p-1)\rho(t)}\left(\|\nabla (u+t\phi)\|^p_p-\|\nabla u\|^p_p\right)\\
	&- e^{-2\rho(t)}\int_{\R^2}\Big(F\big(e^{\rho(t)}(u+t\phi)\big)-F\big(e^{\rho(t)}u\big)\Big)\dif x\\
	=&\frac{1}{2}e^{2\rho(t)}\int_{\R^2}\left(2t\nabla u \nabla \phi +t^2|\nabla \phi|^2\right)\dif x-e^{-2\rho(t)}\int_{\R^2}f\big(e^{\rho(t)}(u+\eta_tt\phi)\big)e^{\rho(t)}t\phi\dif x\\
	&+te^{2(p-1)\rho(t)}\int_{\R^2}\abs{\nabla (u+\theta_t t\phi)}^{p-2}\nabla (u+\theta_t t\phi)\nabla \phi\dif x,
\end{align*}
with $\eta_t, \theta_t\in (0,1)$ and similarly,
\begin{align*}
	&I(\rho(t) \star(u+t\phi))-I(\rho(0)\star u)\geq I(\rho(0)\star(u+t\phi))-I(\rho(0)\star u)\\
	=&\frac{1}{2}e^{2s_u}\int_{\R^2}\left(2t\nabla u\nabla \phi +t^2|\nabla \phi|^2\right)\dif x-e^{-2s_u}\int_{\R^2}f\big(e^{s_u}(u+\zeta_t t\phi)\big)e^{s_u}t\phi\dif x \\
	&+te^{2(p-1)s_u} \int_{\R^2} \abs{\nabla (u+\tau_t t\phi)}^{p-2}\nabla (u+\tau_tt\phi) \nabla \phi \dif x,
\end{align*}
where $\zeta_t,\tau_t\in (0,1)$. In light of Lemma $\ref{lem2.7}$, we deduce $\lim_{t \to 0}\rho(t)=\rho(0)=s_u$.
We then conclude that
\begin{align*}
	&\lim_{t\to 0}\frac{\varPsi(u+t\phi)-\varPsi(u)}{t}&\\
	&=e^{2s_u}\int_{\R^2}\nabla u \nabla \phi\dif x+e^{2(p-1)s_u}\int_{\R^2}\abs{\nabla u}^{p-2}\nabla u\nabla \phi\dif x -e^{-2s_u}\int_{\R^2}f(e^{s_u}u)e^{s_u}\phi\dif x.
\end{align*}
This means that the G$\hat{\text{a}}$teaux derivative of $\varPsi$ is linearly bounded with respect to $\phi$ and continuous in $u$. Therefore,  $\varPsi$ is of class $C^1$, see e.g., \cite{W96}.
Consequently, the conclusion follows immediately upon performing a change of variables in the integrals.
\end{proof}

We then consider the constrained functional $J: \cS_m\to \R$ defined by
$J:=\varPsi|_{\cS_m}$.
Let $T_u\cS_m$ denote the tangent space at $u$ on the manifold $\cS_m$. As a direct consequence of Lemma \ref{comp1}, we have the following corollary.
\begin{corollary}\label{cor1}
The functional $J: \cS_m\to \R $ is of class $C^1$. Moreover, for $u\in \cS_m$ and $\phi\in T_u\cS_m$, one has
\begin{align*}
	dJ(u) \phi=d\varPsi(u)\phi=dI(s_u\star u)(s_u\star \phi).
\end{align*}

\end{corollary}

We now recall the minimization problem \eqref{minipoho}
\begin{equation*}
\gamma_m=\inf_{u\in \mathcal{P}_m}I(u).
\end{equation*}
As a consequence of Lemma \ref{lem2.7}, we are readily have   $0<\varUpsilon_{m}\leq \gamma_m <+\infty$.
\begin{lemma}\label{Ps1}
Assume that $f$ satisfies $(f_1)$-$(f_4)$. Then there exists a Palais--Smale sequence $\{v_n\}\subset \cP_m$ for the constrained functional $I|_{\cS_m}$ at the level $\gamma_m$ for all $\xi\geq \bar{\xi}_0$.
\end{lemma}
\begin{proof}
Let $\mathcal{G}$ be the class of all singletons included in $\cS_m$. According to the terminology in \cite{Gh93}*{Section 3}, we note that $\mathcal{G}$ is a homotopy-stable family of compact subsets of $\cS_m$ with $B=\varnothing$. We denote
\begin{equation*}E_{m,\mathcal{G}}:=\inf_{A\in \mathcal{G}}\max_{u\in A}J(u).
\end{equation*}
Then by Lemma \ref{lem2.7}, for any $A=\{u\}\in \mathcal{G}$, we have $\max_{u\in A}J(u)=I(s_u\star u)$. Note that $s_u\star u\in \cP_m$ for any $u\in \cS_m$.We then conclude that $E_{m,\mathcal{G}}\geq \gamma_m$.
On the other hand, for any $u\in \cP_m$, we derive again from Lemma \ref{lem2.7} that $s_u=0$, so that
$I(u)=I(s_u\star u)\geq E_{m,\mathcal{G}}$. Therefore, we  readily prove that $E_{m, \mathcal{G}}= \gamma_m >0$.

Let $\{A_n\}\subset\mathcal{G}$ be an arbitrary minimizing sequence of $E_{m, \mathcal{G}}$. We define a map
\begin{align*}
	\eta: [0,1]\times \cS_m \to \cS_m,\quad \eta(t,u)=(ts_u)\star u.
\end{align*}
In light of Lemma \ref{lem2.7}, the map $\eta$ is well-defined and continuous, confirming that
$\eta$ is a homotopy on $\cS_m$.
In view of $\eta(t,u)=u$ for all $(t,u)\in \{0\}\times \cS_m$, together with the definition of a homotopy-stable family of $\mathcal{G}$, 	we deduce that
\begin{align*}
	K_n:=\eta(1,A_n)=\{s_u\star u : u\in A_n\}\in \mathcal{G}.
\end{align*}
Clearly, $K_n\subset \cP_m $ for every $n\in \mathbb{N}^+$ and so that we have $J(u)=I(u)$ for any $u\in K_n$.
By virtue of  Lemma  \ref{lem2.7} and the fact that $I(s\star(t\star u))=I((s+t)\star u)$, one has $s_{(t\star u)}=s_u-t$ for any $t\in \R$ and $u\in \cS_m$. It then follows that
\begin{align*}
	J(t\star u)=I(s_{(t\star u)}\star(t\star u))=I((s_u-t)\star (t\star u))=I(s_u\star u)=J(u).
\end{align*}
Note that $\max_{K_n}I(u)=\max_{K_n}J(u)=\max_{A_n}J(u)\to E_{m,\mathcal{G}}=\gamma_m$. We  conclude that $\{K_n\}$ is a minimizing sequence of $\gamma_{m}$. Concerning Lemma \ref{lem3.3} and Remark \ref{rem34}, we know that $\{K_n\}\subset \cP_m$ is uniformly bounded for all $\xi\geq \bar{\xi}_0$. Therefore, employing the minimax principle \cite{Gh93}*{Theorem 3.2}, we can obtain a Palais--Smale sequence $\{r_n\}\subset \cS_m$ for $J$ at the level $E_{m,\mathcal{G}}$ such that $\dist_{E}(r_n,K_n)\to 0$ as $n\to \infty$. Furthermore, up to a subsequence, we deduce that \begin{equation*}\sup_{n\geq 1}\|r_n\|_{E}<+\infty.
\end{equation*}

Let $v_n:=s_n\star r_n$ with $s_n:=s_{r_n}$. It is clear that $\{v_n\}\subset \cP_m$. We then obtain, as $n\to \infty$,
\begin{align*}
	I(v_n)=I(s_n\star r_n)=\Psi(r_n)=J(r_n)\to E_{m, \mathcal{G}}.
\end{align*}
We conclude the proof by showing that $dI(v_n)\to 0$ strongly in the dual space of $T_{v_n}\cS_m$. To this end,
we claim first that there exists $C_0>0$ such that $e^{-s_n}\leq C_0$ for every $n$. For the sake of clarity, the proof of this assertion will be deferred. Note that for any $\varphi\in T_{v_n}\cS_m$, we easily check that
\begin{align*}
	\int_{\R^2}r_n((-s_n)\star\varphi)\dif x=\int_{\R^2}(s_n\star r_n)\varphi\dif x=\int_{\R^2}v_n\varphi\dif x=0,
\end{align*}
which implies $(-s_n)\star \varphi \in T_{r_n}\cS_m$. From the above claim, there exists $C_1>0$ independent of $n$ such that
\begin{equation*}
\|(-s_n)\star\varphi\|^2_E=\|\varphi\|^2_2+e^{-2s_n}\|\nabla \varphi\|^2_2+e^{(\frac{4}{p}-4)s_n}\|\nabla \varphi\|^2_p\leq C^2_1\|\varphi\|^2_E.
\end{equation*}
Thereby, we derive from Corollary $\ref{cor1}$ that
\begin{align*}
	\|dI(v_n)\|_{v_n,*}&
	=\sup_{\substack{\varphi\in T_{v_n}\cS_m \\ \|\varphi\|_E\leq 1}}|dI(s_n\star r_n)(s_n\star((-s_n)\star\varphi))|&\\
	&=\sup_{\substack{\varphi\in T_{v_n}\cS_m \\  \|\varphi\|_E\leq 1}}|dJ(r_n)((-s_n)\star \varphi)|&\\
	&\leq \|dJ(r_n)\|_{r_n,*}\sup_{\substack{\varphi\in T_{v_n}\cS_m \\ \|\varphi\|_E\leq 1}}\|(-s_n)\star \varphi\|_E
	 \leq C_1\|dJ(r_n)\|_{r_n,*},
\end{align*}
which implies $	\|dI(v_n)\|_{v_n,*}\to 0$ as $n\to\infty$.

Finally, we return to prove the claim. To this end, we show that $\{\|\nabla v_n\|_2\}$ is bounded from below by a positive constant. Suppose by contradiction that $\|\nabla v_n\|_2\to 0$ as $n\to \infty$. In fact, by  $(f_1)$-$(f_2)$, for given $\alpha>\alpha_0$,  and $\varepsilon>0$,  there exists  $C_{\varepsilon}>0$ such that,
\begin{align*}
	F(t)\leq f(t)t \leq \varepsilon\abs{t}^{4}+C_{\varepsilon}\abs{t}^{4}\big(e^{\alpha\abs{t}^2}-1\big) ,  \quad \forall s\in \mathbb{R}.
\end{align*}
It is easy to check that $\|v_n\|^2_2=m$ and $\|\sqrt{\alpha/\pi}\nabla v_n\|_2\leq 1$ for sufficiently large $n$.
From Lemma \ref{lem2.1}, we therefrom obtain a constant $C_2>0$ such that
\begin{align*}
	\int_{\mathbb{R}^2}\big(e^{2\alpha|v_n|^2}-1\big)\dif x\leq \int_{\mathbb{R}^2}\big(e^{2\pi|\sqrt{\alpha/\pi} v_n|^2}-1\big)\dif x\leq C^2_2.
\end{align*}
Combining the  H\"{o}lder inequality with the
Gagliardo--Nirenberg interpolation inequality in $H^1(\R^2)$, there exists $C_3, C_4>0$  such that
\begin{align*}
	\int_{\mathbb{R}^2} H(v_n) \dif x&\leq3\varepsilon\int_{\mathbb{R}^2}|v_n|^4 \dif x+3C_{\varepsilon}\int_{\mathbb{R}^2}|v_n|^4\big(e^{\alpha|v_n|^2}-1\big)\dif x\\
	&\leq 3\varepsilon mC_3\|\nabla v_n\|^{2}_2+3C_{\varepsilon}\|u\|_{8}^4 \left(\int_{\mathbb{R}^2}\big(e^{2\alpha \abs{v_n}^2}-1\big)\dif x\right)^{1/2}\\
	&\leq 3\varepsilon mC_3\|\nabla v_n\|^{2}_2+3mC_{\varepsilon}C_2 C_4\|\nabla v_n\|_{2}^3\leq \frac{1}{2}\|\nabla v_n\|^2_2.
\end{align*}
We then conclude that for sufficiently large $n$,
\begin{align*}
	0=P(v_n)&=\|\nabla v_n\|^2_2+2\big(1-\frac{1}{p}\big)\|\nabla v_n\|^p_p-\int_{\R^2}H(v_n)\dif x\\
	&\geq \frac{1}{2}\|\nabla v_n\|^2_2+2\big(1-\frac{1}{p}\big)\|\nabla v_n\|^p_p>0,
\end{align*}
which is a contradiction. Note that $\sup_{n}\|r_n\|_{E}<+\infty$. It then follows that there exists $C_0>1$ independent of $n$ such that
\begin{align*}
	e^{-s_n}=\left(\frac{\int_{\R^2}|\nabla r_n|^2\dif x}{\int_{\R^2}|\nabla v_n|^2\dif x}\right)^{\frac{1}{2}}\leq C_0.
\end{align*}
The claim follows.
\end{proof}

\noindent {\bf Proof of Theorem \ref{thm1}}. Repeating the arguments of Lemma \ref{Ps1} within the function space $E_r$,  and noting Remark \ref{rem34}, we find that there exists a Palais--Smale sequence $\{v_n\}\subset \cP_m\cap E_r$ for the functional $I$ such that, as $n\to\infty$,
\begin{equation*}
I(v_n)\to \gamma_m=\gamma_{m,r}\; \text{ and } \; I|_{\cS_m}'(v_n)\to 0, \text{ strongly in } E^*_r.
\end{equation*}  Combining Lemma \ref{equalvalues} with Remark \ref{rem34}, we conclude that for sufficiently large $\xi$,  \begin{equation*}\gamma_{m,r}=\gamma_m=\varUpsilon_{m}=\varUpsilon_{m,r}.\end{equation*}
As a result, $\{v_n\}\subset \cP_m\cap E_r\subset \cB_m\cap E_r$ is a special minimization sequence of $\varUpsilon_{m}$. By employing Lemma \ref{equalvalues}  once more, we conclude that  there exists $\xi_0>0$ and $v\in \cS_m\cap E_r$ such that $v_n\to v$ strongly in $E$ and $I(v)=\varUpsilon_{m}>0$ for all $\xi\geq \xi_0$.
Recalling that $\| d J(r_{n})\| _{r_{n,*}}\to 0$ strongly, by \cite{BL2}*{Lemma 3}, there exists $\lambda_n\in \R$ such that
\begin{equation*}dJ(r_n)+\lambda_nr_n=o(1).\end{equation*}
Here $o(1)$ denote an infinitesimal in $E_r^*$. Note that $s_{r_n}\star \phi\in T_{v_n}\cS_m$ for any $\phi\in T_{r_n}\cS_m$ and both of $\{s_n\}$ and $\{\|r_n\|_E\}$ are bounded. We therefore deduce by Corollary \ref{cor1} that \begin{equation*}I'(v_n)+\lambda_n v_n=o(1),\end{equation*}
that is, we have
\begin{equation*}-\Delta v_n-\Delta_p v_n +\lambda_n v_n-f(v_n)\to 0, \quad \text{ strongly  in  } E_r^*,\end{equation*}
which gives us that
\begin{equation*}
\int_{\R^2}\abs{\nabla v_n}^2 \dif x +\int_{\R^2}\abs{\nabla v_n}^{p}\dif x+\lambda_n \int_{\R^2}\abs{v_n}^2\dif x=\int_{\R^2} f(v_n)v_n\dif x+o_n(1).
\end{equation*}
This, together with the fact that $v_n\to v$ strongly in $E$, implies that the boundedness of $\{\lambda_n\}$. In fact, we have
\begin{equation*}
	\lambda_n m=\int_{\R^2} f(v_n)v_n\dif x-\int_{\R^2}\abs{\nabla v_n}^2\dif x -\int_{\R^2}\abs{\nabla v_n}^{p}\dif x.
	\end{equation*}
Up to a subsequence, we assume that $\lambda_n\to \lambda\in \R$. By the principle of symmetric criticality \cite{Palais79}, we then conclude that $v\in \cS_m\cap E_r$ with $I(v)=\gamma_m$ satisfying
\begin{align*}
-\Delta_p v-\Delta v+\lambda v=f(v),  \quad  x\in \mathbb R^2.
\end{align*}
 Consequently, we have
\begin{equation*}
\|\nabla v\|^2_2+\|\nabla v\|_p^p+\lambda \|v\|_2^2=\int_{\R^2}f(v)v \dif x.
\end{equation*}

We finally conclude that  $\lambda_\xi>0$ for all $\xi\geq \xi_0$.
To be more clear, we use the notation $v_{\xi}:=v$ and $\gamma_m(\xi):=\gamma_m$ to exhibit the dependence of $v$ and $\gamma_m$ on
$\xi$, respectively. By contradiction, we suppose there exists $\xi_k\to+\infty$ such that $\lambda_{\xi_k}\leq 0$.
Note that for every $k\in \mathbb N$, $v_{\xi_k}\in \cP_m$ and $I(v_{\xi_k})=\gamma_m(\xi_k)$. We thus obtain
\begin{align}\label{posilagmultip}
\lambda_{\xi_k} m+2\gamma_m(\xi_k)
	=\|\nabla v_{\xi_k}\|_2^2+\|\nabla v_{\xi_k}\|_p^p.
\end{align}
It then follows that $\{\lambda_{\xi_k}\}$ is bounded in $\mathbb R$.  Without loss of generality, we may assume that  $\lambda_{\xi_k} \to\lambda_0$ as $k\to \infty$, which implies that $\lambda_0\leq 0$. By Lemma \ref{lem3.5}, we see that  $\gamma_{m}(\xi_k)\to 0^+$ as $k\to\infty$. If $\lambda_0<0$, then for sufficiently large $k$,  we have $\lambda_{\xi_k}\leq \lambda_0/2$,  which, together with \eqref{posilagmultip} and $\gamma_m(\xi_k)\to 0$ as $k\to\infty$, implies that $\|v_{\xi_k}\|_{2,p}<0$, this is impossible.  On the other hand,  if  $\lambda_0=0$, from \eqref{posilagmultip} and Lemma \ref{lem3.5} again, we obtain $\|v_{\xi_k}\|_{2,p}\to 0$ as $k \to \infty$. This would also lead to a contradiction once we apply the arguments from \eqref{upperbound} to \eqref{lowerbound}.
 \qedsymbol

\vskip.2in
\noindent{\bf Funding }
\noindent
This work was partially supported by
NSFC 12271436 and 12371119.
\vskip.1in
\noindent{\bf Data availability } Data sharing is not applicable to this article, as no datasets were generated or analyzed during the current study.
\vskip.1in

\noindent{\bf Conflict of interest} The authors declare that they have no financial or proprietary interests in any material discussed in this article. On behalf of all authors, the corresponding author states that there are no conflicts of interest.
 \vskip.1in


\bigskip
\addcontentsline{toc}{section}{References}
 

\end{document}